\documentclass[11pt]{amsart}

\usepackage{amsmath,amsthm,amssymb}

\usepackage[colorlinks,
    linkcolor={purple!50!black},
    citecolor={blue!50!black},
    urlcolor={blue!80!black}]{hyperref}
\usepackage{mathtools}

\usepackage{graphicx,xcolor}
\usepackage{pinlabel}
\usepackage{float}

\usepackage[all]{xy}

\usepackage[top=1.25in, bottom=1.25in, left=1.25in, right=1.25in]{geometry}

\usepackage{caption}
\usepackage{enumerate}

\theoremstyle{plain}
\newtheorem{proposition}{Proposition}[section]
\newtheorem{theorem}[proposition]{Theorem}

\newtheorem*{theorem*}{Theorem}

\newtheorem*{thm_4.3}{Theorem 4.3}
\newtheorem*{thm_3.1}{Theorem 3.1}
\newtheorem*{thm_5.3}{Theorem 5.3}

\newtheorem*{theoremA}{Theorem A}
\newtheorem*{theoremB}{Theorem B}
\newtheorem*{theoremC}{Theorem C}

\newtheorem{theoremalpha}{Theorem}
\newtheorem{corollary}[proposition]{Corollary}
\newtheorem{lemma}[proposition]{Lemma}
\newtheorem{question}[proposition]{Question}

\theoremstyle{definition}
\newtheorem{definition}[proposition]{Definition}
\newtheorem{conjecture}[proposition]{Conjecture}
\newtheorem{remark}[proposition]{Remark}
\newtheorem{example}[proposition]{Example}
\DeclareMathOperator{\spn}{span}
\DeclareMathOperator{\lk}{lk}
\DeclareMathOperator{\im}{im}
\DeclareMathOperator{\Id}{Id}
\DeclareMathOperator{\id}{Id}
\DeclareMathOperator{\ab}{ab}
\DeclareMathOperator{\pr}{pr}
\DeclareMathOperator{\lcm}{lcm}
\DeclareMathOperator{\diag}{diag}
\DeclareMathOperator{\fix}{fix}

\newcommand{\bp}{\begin{pmatrix}}
\newcommand{\ep}{\end{pmatrix}}
\newcommand{\be}{\begin{equation}}
\newcommand{\ee}{\end{equation}}
\newcommand{\ol}[1]{\overline{#1}}

\newcommand{\smfrac}[2]{\mbox{\footnotesize$\displaystyle\frac{#1}{#2}$}} 
\newcommand{\tmfrac}[2]{\mbox{\large$\frac{#1}{#2}$}} 

\def\Z{\mathbb Z}

\def\Q{\mathbb Q}

\def\wt#1{\widetilde{#1}}

\def\sm{\smallsetminus}

\def\a{\alpha}

\def\AK{\mathcal{A}^{\Q}(K)}
\def\Bl{\mathcal{B}\ell^{\Q}}

\def\toiso{\xrightarrow{\simeq}}

\def\bp{\begin{pmatrix}}
\def\ep{\end{pmatrix}}
\def\ba{\begin{array}}
\def\ea{\end{array}}
\def\bn{\begin{enumerate}}
\def\en{\end{enumerate}}

\DeclareMathOperator\Arf{Arf}

\DeclareMathOperator{\Inn}{Inn}

\begin{document}
\title[A ribbon obstruction and derivatives of knots]{A ribbon obstruction and derivatives of knots}

\author{JungHwan Park}
\address{Max-Planck-Institut f\"{u}r Mathematik}
\email{jp35@mpim-bonn.mpg.de }
\urladdr{http://people.mpim-bonn.mpg.de/jp35}

\author{Mark Powell}
\address{Department of Mathematical Science, Durham University, United Kingdom}
\email{mark.a.powell@durham.ac.uk}

\def\subjclassname{\textup{2010} Mathematics Subject Classification}
\expandafter\let\csname subjclassname@1991\endcsname=\subjclassname
\expandafter\let\csname subjclassname@2000\endcsname=\subjclassname
\subjclass{%
 57M25, 
 57M27, 
 57N70, 
}
\keywords{Slice, doubly slice, homotopy ribbon, Milnor's invariants, derivatives of knots}

\begin{abstract}
We define an obstruction for a knot to be $\Z[\Z]$-homology ribbon, and use this to provide restrictions on the integers that can occur as the triple linking numbers of derivative links of knots that are either homotopy ribbon or doubly slice.  Our main application finds new non-doubly slice knots.  In particular this gives new information on the doubly solvable filtration of Taehee Kim: doubly algebraically slice ribbon knots need not be doubly $(1)$-solvable, and doubly algebraically slice knots need not be $(0.5,1)$-solvable.  We also discuss potential connections to unsolved conjectures in knot concordance, such as generalised versions of Kauffman's conjecture.  Moreover it is possible that our obstruction could fail to vanish on a slice knot.
\end{abstract}

\maketitle

\section{Introduction}\label{Introduction}

Consider a set of curves on a Seifert surface for an algebraically slice knot in $S^3$ that represent a basis for a metaboliser of the Seifert form, namely a half--rank summand of the homology of the surface on which the form vanishes.  Such a set of curves on a Seifert surface, considered as a link in $S^3$ in its own right, is called a \emph{derivative} of $K$.  Derivatives are highly non-unique.  Note that if a knot has a slice derivative link, then the knot is itself slice, since the slicing discs can be used to surger the Seifert surface in the 4-ball to a disc.  One is led to consider the converse.  In other words, we would like to understand, when a knot $K$ has a non-slice derivative link, the situations in which we can deduce that $K$ is not slice.

In the literature there are several higher order signature obstructions, which use a non-vanishing signature of a derivative link to deduce that the original knot is not slice, for example \cite{Cooper82, Gi83, GL92b, Gil93, COT04, CHL10, GL13, Burke14}.  It is an interesting question to determine the extent to which other concordance invariants of links can be applied in this manner.

Towards this end, in this article we study Milnor's triple linking numbers of derivative links.  For an oriented link $L$, the triple linking numbers $\ol{\mu}_L(ijk)$ were one of the first known link invariants that need not vanish on links with unknotted components and vanishing linking numbers. For links with vanishing linking numbers, they are integers.   We provide restrictions on the integers that can arise as the triple linking numbers $\ol{\mu}_L(ijk)$ of derivative links if the base knot $K$ is a homotopy ribbon or a doubly slice knot.

In this paper, knots and links will always come with a choice of orientation.  Recall that a knot $K$ is \emph{slice} if it is the boundary of some  locally flat embedded disc in the four ball $D^4$, \emph{homotopy ribbon} if there exists a slicing disc for which the fundamental group of the knot exterior surjects onto the fundamental group of the slice disc exterior, and $K$ is \emph{doubly slice} if it occurs as an equatorial cross section $K = S \cap S^3$ of an unknotted locally flat 2-sphere $S \subset S^4$ embedded in the 4-sphere, and so slices in two different ways.

\subsection{Doubly slice knots}

Theorem~\ref{theorem:doubly-slice-theorem} below constructs new families of algebraically doubly slice but not doubly slice knots.  They are detected by virtue of nonzero Milnor triple linking numbers of derivatives links.  The new properties of our knots can be expressed in terms of Taehee Kim's doubly-solvable filtration \cite{Kim06}, of the set of knots up to concordance $\mathcal{C}$, by submonoids $\{\mathcal{F}_{n,m}\}$, where $n,m \in \tmfrac{1}{2}\mathbb{N}_0$.  This filtration generalises the solvable filtration of~\cite{COT03}.  Roughly speaking, for those familiar with the solvable filtration, a knot $K$ is $(n,m)$-solvable if the zero-framed surgery $M_K$ bounds an $n$-solution $W_n$ and an $m$-solution $W_m$ such that the union $W_n \cup_{M_K} W_m$ has fundamental group $\Z$.  Sometimes we say doubly $(n)$ solvable instead of $(n,n)$-solvable.

The pertinent facts are the following.
\begin{enumerate}[(i)]
{\setlength\itemindent{15pt}\item Doubly slice knots are $(n,m)$-solvable for all $n,m \in \tmfrac{1}{2}\mathbb{N}_0$~\cite{Kim06}.}
{\setlength\itemindent{15pt}\item  Algebraically slice is equivalent to $(0.5)$-solvable \cite{COT03}, which is itself equivalent to doubly $(0.5)$-solvable, since any algebraically slice knot admits a $(0.5)$-solution with fundamental group $\Z$.}
{\setlength\itemindent{15pt}\item A knot is algebraically doubly slice if a Seifert form admits two dual metabolisers.  Every doubly $(1)$-solvable knot is algebraically doubly slice.  Of course every doubly algebraically slice knot is algebraically slice (Proposition~\ref{prop:kim06}).}
{\setlength\itemindent{15pt}\item Every homotopy ribbon knot is $(n,0.5)$-solvable for all $n$ (Lemma~\ref{lemma:0.5ribbon}).}
\end{enumerate}

\noindent Here is our first main result.

\begin{theoremalpha}\label{theorem:doubly-slice-theorem}~
\begin{enumerate}[(a)]
{\setlength\itemindent{15pt}\item\label{item-theorem-A-1} There exists a ribbon knot that is algebraically doubly slice, but not doubly $(1)$-solvable.}
{\setlength\itemindent{15pt}\item\label{item-theorem-A-2} There exists a knot that is algebraically doubly slice, but not $(0.5,1)$-solvable.}
\end{enumerate}
In particular, neither knot is doubly slice.
\end{theoremalpha}

This has the consequence that algebraically doubly slice does not correspond precisely to any step in the doubly solvable filtration.
To prove Theorem~\ref{theorem:doubly-slice-theorem}, we construct knots with derivatives having nonvanishing triple linking numbers, and we show that these triple linking numbers cannot occur for a doubly slice knot, nor indeed for a doubly $(1)$-solvable knot.

\subsection{A homology ribbon obstruction}

To state our obstruction theorem we introduce the following notion.

\begin{definition}\label{defn:homology-ribbon}
  A knot $K$ is said to be \emph{$\Z[\Z]$-homology ribbon} if there is a slice disc $D \subset D^4$ for $K$ such that the induced map $H_1(S^3 \sm \nu K ;\Z[\Z]) \to H_1(D^4 \sm \nu D;\Z[\Z])$ is surjective.
\end{definition}

A derivative link of a given knot depends on a choice of Seifert surface, a choice of metaboliser, and a choice of curves representing that metaboliser.  In order to obtain an obstruction that is independent of choices,  we take the fundamental class $[M_K] \in H_3(M_K;\Z)$ of the zero-framed surgery manifold $M_K$, and map it to the group homology $H_3(B\Gamma(K,P);\Z)$, where $\Gamma(K,P)$ is a group that depends on the knot $K$ and a lagrangian $P$ for the rational Blanchfield form of $K$. The map $M_K \to B\Gamma(K,P)$ arises from a representation $\a_P \colon \pi_1(M_K) \to \Gamma(K,P)$ that also depends on $P$.  The image of $[M_K]$ in $H_3(B\Gamma(K,P);\Z)$ is our obstruction $\psi(K,P)$.

\begin{theorem}\label{thm:obstruction-intro}
 Suppose that a knot $K$ is $\Z[\Z]$-homology ribbon.  Then there is a lagrangian $P$ for the rational Blanchfield form such that $\psi(K,P) = 0 \in H_3(B\Gamma(K,P);\Z)$.
\end{theorem}

One then observes that if $K$ is homotopy ribbon, it is $\Z[\Z]$-homology ribbon.  On the other hand, if $K$ is doubly slice, then it is $\Z[\Z]$-homology ribbon in two ways. Thus our obstruction for a knot to be $\Z[\Z]$-homology ribbon can be applied to obstruct a knot from being doubly slice and homotopy ribbon.  As far as we know, doubly slice knots need not be homotopy ribbon, and there are homotopy ribbon knots that are not doubly slice, so Definition~\ref{defn:homology-ribbon} is necessary to unify the treatment.

\begin{theorem}\label{thm-intro-homotopy-ribbon}
Suppose that a knot $K$ lies in $\mathcal{F}_{0.5,1}$ $($for example if $K$ is homotopy ribbon$)$. Then there is a lagrangian $P$ for the rational Blanchfield form such that $\psi(K,P)=0$.
\end{theorem}

\begin{theorem}\label{thm-intro-doubly-slice}
Suppose that a knot $K$ lies in $\mathcal{F}_{1,1}$ $($for example if $K$ is doubly slice$)$.  Then there are lagrangians $P_1$ and $P_2$ for the rational Blanchfield form such that $P_1 \oplus P_2 = H_1(M_K;\Q[\Z])$ and $\psi(K,P_1) = \psi(K,P_2) = 0$.
\end{theorem}

Versions of these obstructions have been known to the experts for some time; we learnt about them from Tim Cochran, Shelly Harvey, Kent Orr and Peter Teichner.
We had to deal precisely with differences between rational and integral Alexander modules, and the relationship between Blanchfield and Seifert forms, in order to make the obstruction practical.  Another new ingredient now is the work of the first author \cite{Park16}, which gives a procedure to obtain infinitely many different integers as the triple linking numbers $\ol{\mu}(123)$ of derivative links representing a fixed set of homology classes on a Seifert surface.

\subsection{Determining the possible triple linking numbers of derivatives}

Let $K$ be a knot with a genus three Seifert surface $\Sigma$, and let $H \subset H_1(\Sigma;\Z)$ be a metaboliser for the Seifert form of $\Sigma$.
We write $\mathfrak{d} K /\mathfrak{d}H$ for the set of all derivative links on $\Sigma$ whose homology classes span $H$.  We consider a derivative link as an ordered and oriented link.  Since $H$ is a rank three free abelian group, its third exterior power $\bigwedge^3 H \cong \Z$.  Let $\mathfrak{o}(L)$ be the generator $[L_1] \wedge [L_2] \wedge [L_3] \in \bigwedge^3 H$.
 We investigate the set
\[S_{K,H} := \{ \bar{\mu}_{L}(123) - \bar{\mu}_{L'}(123) \mid L,L' \in \mathfrak{d}K/\mathfrak{d}H,\,\mathfrak{o}(L) = \mathfrak{o}(L')\}.\]
Note that this set is for a fixed Seifert surface; a priori it could vary for different Seifert surfaces.
For certain knots and certain Seifert surfaces we are able, in combination with the results of \cite{Park16}, to determine this set precisely.
Here is our second main result.

\begin{theoremalpha}\label{thm:determination-of-mu-123}
  Let $K$ be a knot that admits a genus three Seifert surface $\Sigma$ that has a basis for the first homology $H_1(\Sigma;\Z)$, with respect to which the Seifert form is given by  $\begin{pmatrix}
    A & X \\ X -I & 0
  \end{pmatrix}$,
where $X = \diag(p_1,p_2,p_3)$ is a diagonal matrix and there are no restrictions on~$A$.  Write $n := \det(X) -\det(X-\Id)$.  Suppose that the nonzero entries $p_i$ are such that $\gcd(p_i,n)=\gcd(p_i-1,n)=1$, and $p_i\cdot(p_i-1) \neq 0$ for $i =1,2,3$.
Let $H$ be the metaboliser generated by the last three basis elements of $H_1(\Sigma;\Z)$.
Then $S_{K,H} = n \Z$.
\end{theoremalpha}

The inclusion $n\Z \subseteq S_{K,H}$ was shown by the first author in \cite{Park16}.  In particular \cite[Corollary~$4.5$]{Park16} produced Alexander polynomial one knots having genus 3 Seifert surfaces $\Sigma$, with the property that for every $k \in \Z$ there is a derivative on $\Sigma$ with Milnor triple linking number $k$.  To show the opposite inclusion we employ the obstruction~$\psi$.

As well as understanding the possible Milnor's invariants of derivatives on a fixed Seifert surface, we also exhibit knots for which we can control the Milnor's invariants of all possible derivatives on \emph{all} possible Seifert surfaces.  Recall that a $(0)$-solvable link has all linking numbers and all triple linking numbers vanishing.  We say that a knot is \emph{homotopy ribbon $(1)$-solvable} if there is a $(1)$-solution $W$ with $\pi_1(M_K) \to \pi_1(W)$ surjective. By definition every homotopy ribbon knot is homotopy ribbon $(1)$-solvable.
Here is the third main result of this article.

\begin{theoremalpha}\label{thm:any-seifert-surface-non-trivial-mu-123}
  There exists an algebraically slice knot $K$ that is not homotopy ribbon $(1)$-solvable and moreover does not have any $(0)$-solvable derivative. In particular, for any derivative $J$, there is a subset $\{i,j,k\}$ of the indexing set for the components of $J$ such that $\ol{\mu}_J(ijk) \neq 0$.
\end{theoremalpha}

A knot satisfying Theorem~\ref{thm:any-seifert-surface-non-trivial-mu-123} is constructed by string link infections involving Borromean rings.
Examples of smoothly slice knots that have non-slice derivatives on their unique genus minimising Seifert surface, constructed in \cite{CD14}, suffer from the defect that there exists an unlinked derivative after stabilising. The triple linking numbers in the derivatives of our knots cannot be destroyed by stabilisation.  The big question, of course, is whether any of the knots that we construct for the proof of Theorem~\ref{thm:any-seifert-surface-non-trivial-mu-123} are slice.  More generally, the following question remains open.

\begin{question}
Does every $($smoothly$)$ slice knot have a Seifert surface with a $($smoothly$)$ slice derivative?
\end{question}

It is a standard construction (see for example \cite[Corollary~7.4]{CD14}) that every ribbon knot has a Seifert surface with an unlinked derivative.
If the knots of Theorem~\ref{thm:any-seifert-surface-non-trivial-mu-123} are not slice, it seems likely that they will also not be $(1)$-solvable, and so would show the nontriviality of the quotient $\mathcal{F}_{0.5}/\mathcal{F}_{1}$ of the solvable filtration.
Note that it was recently shown in \cite{DMOP16} that genus one algebraically slice knots are $(1)$-solvable.


\begin{remark}
  We take this opportunity to mention the paper \cite{JKP14}, which purported to show that there are slice boundary links whose derivative links all have nonvanishing triple linking numbers.  Unfortunately, as pointed out by the first author to the second, there is a mistake in the argument given that the links constructed in \cite{JKP14} are slice.  In particular, the 2-complex $Y \times I$ cannot be embedded in the link complement as claimed on pages 16 and 17 of \cite{JKP14}.
\end{remark}

\subsection{Acknowledgements}
The first author would like to thank his advisors Tim Cochran and Shelly Harvey, and also Christopher Davis for helpful discussions.  The authors are grateful to the Max Planck Institute for Mathematics and the Hausdorff Institute for Mathematics in Bonn. Part of this paper was written while the authors were visitors at these institutes.  The authors respectively thank the Universit\'{e} du Qu\'{e}bec \`{a} Montr\'{e}al and Rice University for excellent hospitality.  The second author was supported by an NSERC Discovery Grant.

\section{Background and notation}\label{Background and notation}

\makeatletter
\providecommand*{\twoheadrightarrowfill@}{%
  \arrowfill@\relbar\relbar\twoheadrightarrow
}
\providecommand*{\twoheadleftarrowfill@}{%
  \arrowfill@\twoheadleftarrow\relbar\relbar
}
\providecommand*{\xtwoheadrightarrow}[2][]{%
  \ext@arrow 0579\twoheadrightarrowfill@{#1}{#2}%
}
\providecommand*{\xtwoheadleftarrow}[2][]{%
  \ext@arrow 5097\twoheadleftarrowfill@{#1}{#2}%
}
\makeatother

\subsection{Notation}\label{Notation}
An $m$-component \emph{link} $L = L_1 \cup \dots \cup L_m$ is the image of an embedding of a disjoint union of $m$ circles into the $3$-sphere.  All links in this paper are ordered and oriented.  A \emph{knot} is a 1-component link. We say that an $m$-component link $L$ is \emph{slice} if the components of $L$ bound $m$ locally flat disjointly embedded $2$-discs $D_1 \cup \dots \cup D_m$ in $D^4$ with $\partial D_i = L_i$.  A knot $K$ is called \emph{doubly slice} if it arises as a cross section $S \cap S^3$ of a locally flat unknotted $2$-sphere $S \subset S^4$.

If in addition the slice discs $D_1 \cup \dots \cup D_m$ are smoothly embedded, then we say that the link $L$ is \emph{smoothly slice}.  Similarly, if the 2-sphere $S$ is smoothly embedded then we say that the knot $K$ is \emph{smoothly doubly slice}.

If a smoothly slice knot $K$ bounds a smoothly embedded $2$-disc $D$ in $D^4$ for which there are no local maxima of the radial function restricted to $D$, we call $D$ a \emph{ribbon disc} and we call $K$ a \emph{ribbon knot}. Furthermore, following \cite{CG83}, we say that a knot $K$ is \emph{homotopy ribbon} if there is a slice disc $D$ for which the inclusion induces an epimorphism $\pi_1(S^3 \sm \nu K) \twoheadrightarrow \pi_1(D^4 \sm \nu D)$. Every ribbon knot is homotopy ribbon.

Every knot $K$ in $S^3$ admits a Seifert surface $\Sigma$. Let $g$ be the genus of $\Sigma$.  From $\Sigma$, we can define a Seifert form $\beta_\Sigma \colon H_1(\Sigma ) \times H_1(\Sigma ) \rightarrow \mathbb{Z}$. Levine~\cite[Lemma $2$]{Lev69} showed that if $K$ is a slice knot, then $\beta_\Sigma$ is metabolic for any choice of Seifert surface $\Sigma$ for $K$, that is there exists a direct summand $H \cong \mathbb{Z}^{g}$, of $H_1(\Sigma) \cong \Z^{2g}$, such that $\beta_\Sigma(H \times H)=0$. We call such $H \subset H_1(\Sigma)$ a \emph{metaboliser} of $\beta_{\Sigma}$ and say that a knot $K$ is \emph{algebraically slice} if it has a metaboliser. If $K$ is algebraically slice and $H$ is a metaboliser of $\beta_{\Sigma}$, then a link $J = J_1\cup \cdots \cup J_{g}$ embedded in the Seifert surface $\Sigma$, whose homology classes generate $H$, is called a \emph{derivative} of $K$ associated to $H$. Denote the set of all derivatives associated to $H$ by $\mathfrak{d} K /\mathfrak{d}H$. Note that if a knot $K$ has a (smoothly) slice derivative, then $K$ is (smoothly) slice.

For a link $L$, let $X_L:= S^3\sm \nu L$ denote the exterior of an open tubular neighbourhood of $L$ in $S^3$. Also, for a slice disc $D$, let $D^4 \sm \nu D$ denote the exterior of an open tubular neighbourhood of $D$ in $D^4$. Finally let~$M_L$ be the result of zero-framed surgery on $S^3$ along~$L$.

Denote the set of non-negative integers by $\mathbb{N}_0$ and denote the set of non-negative half integers by $\frac{1}{2}\mathbb{N}_0$.

\subsection{The rational Alexander module and the Blanchfield form}\label{The rational Alexander module and the Blanchfield form}

 Write $G:=\pi_1(M_K)$, let $\mathcal{A}(K)$ be the Alexander module of a knot $K$ and let $\AK$ be the rational Alexander module of $K$. Since the longitudes of $K$ lie in $G^{(2)}$,
$$\mathcal{A}(K) \cong G^{(1)} / G^{(2)} \text{ and } \AK \cong \mathbb{Q}[t,t^{-1}] \otimes_{\mathbb{Z}[t,t^{-1}]} G^{(1)} / G^{(2)}.$$
Here $G^{(k)}$ denotes the $k$th derived subgroup, where $G^{(0)} = G$ and $G^{(k+1)} = [G^{(k)},G^{(k)}]$.
Choose a meridian $g$ for the knot $K$.  Then the abelian group $G^{(1)}/G^{(2)}$ becomes a $\Z[t,t^{-1}]$-module via $t \cdot h := ghg^{-1}G^{(2)}$. As a rational vector space, $\AK$ has rank $d:=\deg \Delta_K(t)$.

The \emph{rational Blanchfield linking form}
\[\mathcal{B}\ell^{\Q} \colon \AK \times \AK \to \Q(t)/\Q[t,t^{-1}]\]
is a nonsingular, hermitian and sesquilinear form.
In addition, a submodule $P \subset \AK$ is called a \emph{lagrangian} if $P=P^{\perp}$, where
$$P^{\perp} := \{ x \in \AK \mid \Bl(x,p) = 0 \text{ for every } p \in P \}.$$
Since $\Bl$ is nonsingular, any lagrangian $P$ has rank $d/2$ as a rational vector space ($d$ is even since the Alexander polynomial of $K$ has a symmetric representative, that is $\Delta_K(t) \doteq \Delta_K(t^{-1})$).

Suppose that $K$ is a slice knot and $D$ is a slice disc, then $\ker(\AK\rightarrow \mathcal{A}^{\Q}(D^4 \sm \nu D))$ is a lagrangian (e.g.\ \cite[Theorem~4.4]{COT03}), where by definition $\mathcal{A}^{\Q}(D^4\sm \nu D) := H_1(D^4 \sm \nu D;\Q[t,t^{-1}])$ is the rational Alexander module of $D^4 \sm \nu D$.  More generally this holds with $D^4 \sm \nu D$ replaced by an $(n)$-solution $W$, for $n \geq 1$ (we will recall the definition of an $(n)$-solution in Section~\ref{section:defn-of-solvable-filtration}). We call $P$ the \emph{lagrangian associated to the slice disc $D$} if $P=\ker(\AK\rightarrow \mathcal{A}^\Q(D^4 \sm \nu D))$.

\begin{remark}
  Following Cochran-Harvey-Leidy~\cite{CHL10}, we use the term lagrangian for a self-annihilating submodule of the rational Alexander module with respect to the rational Blanchfield form, and the term metaboliser for the corresponding object with respect to a Seifert form.
\end{remark}

\subsection{Derivatives of Knots}\label{Derivatives of Knots}

Let $K$ be an algebraically slice knot, let $H$ be a metaboliser of the Seifert form of $K$ with respect to a Seifert surface $\Sigma$, and let $J = J_1\cup \dots \cup J_g$ be a derivative of $K$ associated with $H$, where $g$ is the genus~$\Sigma$. We will use the terminology of \cite[Definition~5.4]{CHL10}.

\begin{definition}\label{definition:Lagrangian}
Suppose that $\Sigma$ is a genus $g$ Seifert surface for $K$ and that $P \subset \AK$ is a lagrangian. We say that the metaboliser $H$ \emph{represents} $P$ if the image of $H$ under the map
$$H_1(\Sigma;\mathbb{Z}) \xhookrightarrow{1 \otimes \Id} \Q \otimes_{\Z} H_1(\Sigma;\mathbb{Z}) \xtwoheadrightarrow{i_*} \AK$$
spans $P$ as a $\mathbb{Q}$-vector space. Note that in order to define $i_*$ we need to fix a lift of $\Sigma$ to the infinite cyclic cover. However, it is easy to check that a metaboliser $H$ represents $P$ with respect to one choice of lift if and only if it represents $P$ with respect to all choices.
\end{definition}

\noindent Next we recall a lemma from \cite[Lemma~$5.5$]{CHL10}.

\begin{lemma}[{Cochran-Harvey-Leidy}]\label{lemma:Lagrangian}
Every lagrangian is represented by some metaboliser.
\end{lemma}

Let $b_i=[J_i]$ in $H_1(\Sigma)$, for $1 \leq i \leq g$. Then we can extend $\{b_1, \dots, b_g\}$ to a symplectic basis $\{a_1, \dots, a_g, b_1, \dots, b_g\}$, where $a_i$ is an intersection dual of $b_i$, for $1 \leq i \leq g$. From this we obtain a disc-band form of $\Sigma$, as depicted in Figure~\ref{figure:diskband}. We need one more proposition from \cite[Proposition~5.6]{CHL10}.


 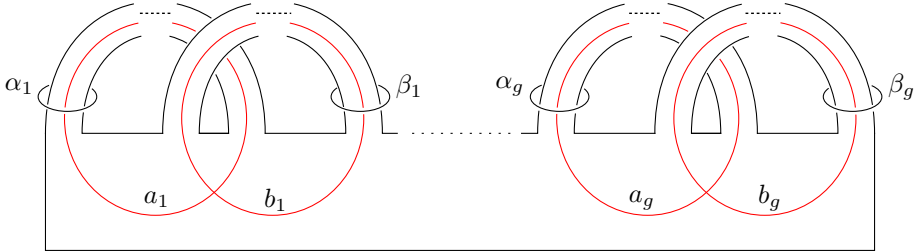
\begin{figure}[ht]
  \labellist\small
  \pinlabel{$\alpha_1$} at -10 90
\pinlabel{$\beta_1$} at 203 90
\pinlabel{$\alpha_g$} at 258 90
\pinlabel{$\beta_g$} at 472 90
\pinlabel{$a_1$} at 65 28
\pinlabel{$b_1$} at 130 28
\pinlabel{$a_g$} at 330 28
\pinlabel{$b_g$} at 400 29
\endlabellist
\includegraphics[scale=0.7]{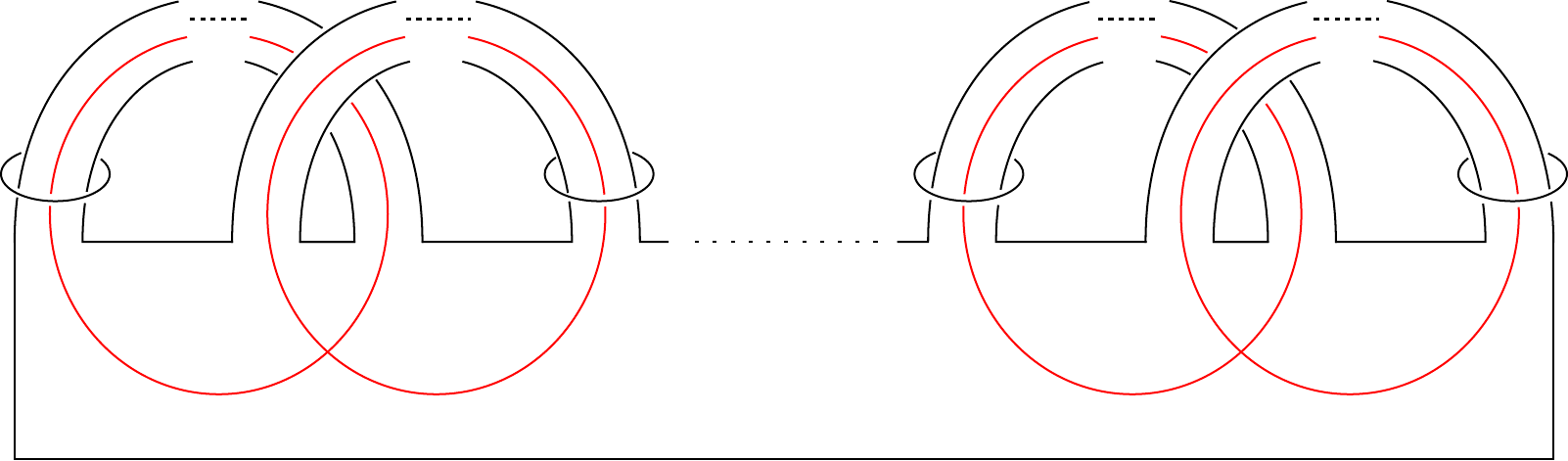}
    \caption{A disc-band form for $\Sigma$, a generating set $\{a_1,b_1,\dots,a_g,b_g\}$ for $H_1(\Sigma;\Z)$ and a dual generating set $\{\alpha_1,\beta_1,\dots,\alpha_g,\beta_g\}$ for $H_1(S^3 \sm \Sigma;\Z)$.}
    \label{figure:diskband}
  \end{figure}

\begin{proposition}[{Cochran-Harvey-Leidy}]\label{proposition:generators}
 Suppose $P \subset \AK$ is a lagrangian. Then for any Seifert surface $\Sigma$, any metaboliser $H$ representing $P$, and any symplectic basis $\{a_1, \dots, a_g, b_1, \dots, b_g\}$ of $H_1(\Sigma;\Z)$ with $\{b_1, \dots, b_g\}$ a basis for $H$, we have:
\begin{enumerate}
\item \label{item:prop-generators-1} The curves $\{b_1, \dots, b_g\}$ span $P$ in the rational vector space $\AK$.
\item \label{item:prop-generators-2} The curves $\{\phi(\beta_1), \dots, \phi(\beta_g)\}$ span $\AK/P$, where $\{\alpha_1, \dots, \alpha_g, \beta_1, \dots, \beta_g\}$ is the basis of $H_1(S^3 \sm \Sigma;\Z)$ dual to $\{a_1, \dots, a_g, b_1, \dots, b_g\}$ under the linking number in $S^3$, and
$$\phi \colon H_1(S^3 \sm \Sigma;\mathbb{Z}) \hookrightarrow \Q \otimes H_1(S^3 \sm  \Sigma;\mathbb{Z}) \xtwoheadrightarrow{i_*} \AK.$$
\end{enumerate}
\end{proposition}

Given a derivative link $J$, we will use Proposition~\ref{proposition:generators} to construct a map $f_J \colon \pi_1(M_J) \rightarrow \mathcal{A}(K)/\widetilde{P}$, where \[\widetilde{P}:=\ker(\mathcal{A}(K) \rightarrow \AK \rightarrow \AK/P).\]

We associate the meridian of the band on which $J_i$ lies ($\beta_i$ of Figure~\ref{figure:diskband}) with a meridian $\mu_i$ of $J_i$.
In order to determine a homotopy class of maps $f_J \colon M_J \rightarrow B(\mathcal{A}(K)/\widetilde{P})$, it suffices to define the image of each meridian $\mu_i \in \pi_1(M_J)$ in $\mathcal{A}(K)/\widetilde{P}$, since any map $\pi_1(M_J) \to \mathcal{A}(K)/\widetilde{P}$ factors through the abelianisation $H_1(M_J;\Z)$.  Send $\mu_i$ to the image of $\beta_i$ under the map $$H_1(S^3\sm \nu \Sigma) \rightarrow \mathcal{A}(K) \rightarrow \mathcal{A}(K)/\widetilde{P}$$
to determine $f_J \colon \pi_1(M_J) \to \mathcal{A}(K)/\widetilde{P}$, and thence (up to homotopy) a map $f_J \colon M_J \to B(\mathcal{A}(K)/\widetilde{P})$.


\subsection{The solvable filtration and the doubly solvable filtration}\label{section:defn-of-solvable-filtration}

We briefly recall the definitions and some basic facts about the solvable filtrations, for the convenience of the reader.

\begin{definition}[{$n$-solvable filtration}]~\label{defn:n-solvable}
 We say that a knot $K$ is \emph{$(n)$-solvable} for $n\in\mathbb{N}_0$ if the zero-framed surgery manifold $M_K$ is the boundary of a compact oriented $4$-manifold $W$ with the inclusion induced map $H_1(M_K;\Z) \to H_1(W;\Z)$ an isomorphism, and such that $H_2(W;\Z)$ has a basis consisting, for some $k$, of $2k$ embedded, connected, compact, oriented surfaces $L_1,\dots ,L_k,D_1,\dots,D_k$ with trivial normal bundles satisfying:
\begin{enumerate}[(i)]
{\setlength\itemindent{15pt}\item    $\pi_1(L_i) \subset \pi_1(W)^{(n)}$ and $\pi_1(D_j) \subset \pi_1(W)^{(n)}$ for all $i,j=1,\dots,k$;}
{\setlength\itemindent{15pt}\item the geometric intersection numbers are $L_i \cdot L_j = 0 = D_i\cdot D_j$ and $L_i \cdot D_j = \delta_{ij}$ for all $i,j =1,\dots,k$.}
\end{enumerate}
Such a $4$-manifold $W$ is called an \emph{$(n)$-solution}.  If in addition $\pi_1(L_i) \subset \pi_1(W)^{(n+1)}$ for all $i$, then $W$ is an \emph{$(n.5)$-solution} and $K$ is \emph{$(n.5)$-solvable}.  The subgroup of $\mathcal{C}$ of $(k)$-solvable knots is denoted $\mathcal{F}_{k}$, for any $k \in \frac{1}{2}\mathbb{N}_0$.
\end{definition}

Note that a slice knot is $(n)$-solvable for all $n$, and the above definition naturally extends to links.
The first two graded quotients of the solvable filtration are well understood. To wit, a knot is $(0)$-solvable if and only if it has vanishing Arf invariant, while a knot is $(0.5)$-solvable if and only if it is algebraically slice. Moreover, the iterated graded quotients of the solvable filtration are all highly non-trivial. In fact, it was shown in \cite{CHL09, CHL11a} that $\mathcal{F}_n/\mathcal{F}_{n.5}$ contains subgroups $\mathbb{Z}^\infty \oplus \mathbb{Z}^\infty_2$ for any $n\in\mathbb{N}_0$. On the other hand, there is not much known about the other quotients, $\mathcal{F}_{n.5}/\mathcal{F}_{n+1}$.   The knots studied in this paper are related to the question of whether $\mathcal{F}_{0.5}/\mathcal{F}_{1}$ is nontrivial. Indeed, we show that the analogous difference is nontrivial in the doubly solvable filtration. We recall the definition of this filtration, due to Taehee Kim~\cite{Kim06}, next.
\medskip

\begin{definition}[{Doubly solvable filtration}]~\label{defn:doubly-n-solvable}
We say that a knot $K$ is \emph{$(n,m)$-solvable}, for $n,m\in\frac{1}{2}\mathbb{N}_0$, if the zero-framed surgery manifold $M_K$ is the boundary of an
$(n)$-solution $W_n$ and an $(m)$-solution $W_m$ such that the fundamental group of the union $W_n \cup_{M_K} W_m$ of $W_n$ and $W_m$ along their boundary is isomorphic to $\mathbb{Z}$. The set of all $(n,m)$-solvable knots is denoted by $\mathcal{F}_{n,m}$. We say that an $(n,n)$-solvable knot is \emph{doubly $(n)$-solvable}.
\end{definition}

\noindent A doubly slice knot is $(n,m)$ solvable for all $n,m\in\frac{1}{2}\mathbb{N}_0$~\cite{Kim06}.  We will frequently use the following fact~\cite[Theorem~4.4]{COT03}.

\begin{lemma}\label{lemma-kernel-metaboliser-for-Bl-1-solutions}
  Let $K$ be a knot with a $(1)$-solution $W$.  Let $P:= \ker\big(H_1(M_K;\Q[\Z]) \to H_1(W;\Q[\Z])\big)$.  Then $P$ is a lagrangian for the rational Blanchfield form of $K$.
\end{lemma}


\noindent We call $P$ the \emph{lagrangian associated to $W$}.

\section{A $\Z[\Z]$ homology ribbon obstruction}\label{Ribbon Obstruction}

In this section we define a homology ribbon obstruction and will introduce some of its properties.  This obstruction will give rise to a homotopy ribbon obstruction, and a doubly slice obstruction.  Our obstruction also works in the context of the solvable filtration, so we will work in this generality.

\begin{definition}\label{defn:homology-ribbon-1-solvable}
  We say that a knot $K$ is \emph{homology ribbon $(1)$-solvable} if there is a $(1)$-solution $W$ with $\partial W = M_K$ such that the inclusion induced map $H_1(M_K;\Z[\Z]) \to H_1(W;\Z[\Z])$ is surjective. Such a $4$-manifold $W$ is called a \emph{homology ribbon $(1)$-solution}.
\end{definition}

\begin{lemma}\label{lemma:homotopyribbonsol}
  Suppose that $K$ is homotopy ribbon $(1)$-solvable $($for example if $K$ is homotopy ribbon$)$.  Then $K$ is homology ribbon $(1)$-solvable.
\end{lemma}

\begin{proof}
  The map on fundamental groups being surjective implies that the map on $\Z[\Z]$-homology is surjective, since $H_1(W;\Z[\Z]) \cong \pi_1(W)^{(1)}/\pi_1(W)^{(2)}$.
\end{proof}

\noindent The following lemma is from \cite[Proposition 2.10]{Kim06}.

\begin{lemma}\label{lemma:kim06}
  Suppose that $K \in \mathcal{F}_{1,1}$ $($for example, every doubly slice knot lies in $\mathcal{F}_{1,1}$$)$.  Then there exist two homology ribbon $(1)$-solutions $W_1$ and $W_2$ such that the inclusion induced maps give rise to an isomorphism
  \[\xymatrix{H_1(M_K;\Z[\Z]) \ar[r]^-{(i_1,i_2)}_-{\cong} & H_1(W_1;\Z[\Z]) \oplus H_1(W_2;\Z[\Z])}\]
   and such that both of the summands become lagrangians for the rational Blanchfield form after tensoring with~$\Q$. In particular, every doubly $(1)$-solvable knot is homology ribbon $(1)$-solvable.
\end{lemma}

The next lemma follows from the proof of {\cite[Proposition 2.10]{Kim06}}.  Since it was not explicitly stated in \cite{Kim06}, we give a quick proof.

\medskip
\begin{lemma}~\label{lemma:kim06half}
Suppose that $K \in \mathcal{F}_{0.5,1}$. Then $K$ is homology ribbon $(1)$-solvable.
\end{lemma}

\begin{proof}
Suppose that $K \in \mathcal{F}_{0.5,1}$, and let $(W_{0.5},W_1)$ be the given $(0.5,1)$-solution pair.  We may and will assume that $\pi_1(W_{0.5}) \cong \Z$ and so $H_1(W_{0.5};\Z[\Z])=0$.   The Mayer-Vietoris sequence for gluing these together contains:
\[H_1(M_K;\Z[\Z]) \to H_1(W_{0.5};\Z[\Z]) \oplus H_1(W_1;\Z[\Z]) \to H_1(W_{0.5} \cup_{M_K} W_1;\Z[\Z]). \]
Since $\pi_1(W_{0.5} \cup_{M_K} W_1) \cong \Z$, it follows that $H_1(W_{0.5} \cup W_1;\Z[\Z])=0$ and so $H_1(M_K;\Z[\Z]) \to H_1(W_1;\Z[\Z])$ is surjective.
Thus $K$ is homology ribbon $(1)$-solvable.
\end{proof}


%
%
%

\noindent We have learnt that obstructions to homology ribbon can be used to obstruct knots from being homotopy ribbon and doubly slice.

\subsection{Definition of the homology ribbon obstruction}

Let $P\subset \AK$ be a lagrangian for the rational Blanchfield form and let
\[\widetilde{P} = \ker(\mathcal{A}(K) \hookrightarrow \AK \rightarrow \AK/P).\]
As above let $G:=\pi_1(M_K)$.  We have a map
$$\phi_P \colon G \rightarrow G/G^{(2)} \toiso \mathbb{Z} \ltimes \mathcal{A}(K) \rightarrow \mathbb{Z} \ltimes \mathcal{A}(K)/\widetilde{P}.$$
Here the identification $\vartheta \colon G/G^{(2)} \toiso \mathbb{Z} \ltimes \mathcal{A}(K)$ depends on a choice of oriented meridian for $K$, which determines a splitting $\theta \colon \Z \to G/G^{(2)}$ of the abelianisation homomorphism.  We have to make such a choice in order to define the invariant that we will introduce below, so we should investigate the dependence of the outcome on this choice.

Write $\Inn(\Gamma)$ for the inner automorphisms of a group $\Gamma$, and for a subgroup $H \leq \Gamma$ write $\Inn_H(\Gamma) \leq \Inn(\Gamma)$ for the subgroup containing conjugations by elements of $H$.

\begin{lemma}\label{lemma:inner-automorphisms}
  Let $\theta_1,\theta_2 \colon \Z \to G/G^{(2)}$ be two choices of splitting as above, and denote the resulting identifications by $\vartheta_1,\vartheta_2 \colon G/G^{(2)} \toiso \Z \ltimes \mathcal{A}(K)$.
   \begin{enumerate}[(i)]
{\setlength\itemindent{15pt}\item  There is an inner automorphism $\gamma \colon \Z \ltimes \mathcal{A}(K) \to \Z \ltimes \mathcal{A}(K)$ in $\Inn_{\mathcal{A}(K)}(\Z \ltimes \mathcal{A}(K))$ such that the $\gamma \circ \vartheta_1 = \vartheta_2$.}
{\setlength\itemindent{15pt}\item  Every inner automorphism $\gamma \in \Inn_{\mathcal{A}(K)}(\Z \ltimes \mathcal{A}(K))$ acts by the identity on $\mathcal{A}(K)$.  In particular it preserves $\widetilde{P} \leq \mathcal{A}(K) \leq G/G^{(2)}$.}
\end{enumerate}
\end{lemma}


\begin{proof}
We claim that we can always arrange that $\theta_1(1) = \theta_2(1)$ up to an inner automorphism.  To see this it suffices to change $\vartheta_1(\theta_2(1))$ to $(1,0) \in \Z \ltimes \mathcal{A}(K)$ by applying an inner automorphism of $\Z \ltimes \mathcal{A}(K)$.  Let $h \in \mathcal{A}(K)$ be such that $\vartheta(\theta_2(1)) = (1,h)$.   By \cite[Proposition~1.2]{Le77}, multiplication by $1-t$ acts as an automorphism of $\mathcal{A}(K)$.
We can therefore find $h' \in H$ such that $(1-t)\cdot h' = h$.  Then we have:
\begin{align*}
(0,h')^{-1}(1,h)(0,h') &= (0,-h')(1,h)(0,h') = (1,-h' + h)(0,h')\\
&= (1,-h' + h + t\cdot h') = (1,h - (1-t)\cdot h') \\ & = (1,h-h) = (1,0).
\end{align*}
Let $\gamma$ be inner automorphism obtained by conjugating with $(0,h') \in \mathcal{A}(K)$.   Then $\gamma \circ \vartheta_1 = \vartheta_2$ as required for the first part of the lemma.  Since $(0,h')$ lies in $\mathcal{A}(K)$, it commutes with $g \in \mathcal{A}(K) \subset \Z \ltimes \mathcal{A}(K)$ for all such $g$.  The second part of the lemma follows.
\end{proof}

Since $\widetilde{P}$ is preserved, an inner automorphism as in Lemma~\ref{lemma:inner-automorphisms} descends to an action on $\mathbb{Z} \ltimes \mathcal{A}(K)/\widetilde{P}$, and acts by the identity on the subgroup $\mathcal{A}(K)/\widetilde{P} \leq \mathbb{Z} \ltimes \mathcal{A}(K)/\widetilde{P}$.

For a choice of splitting $\theta$, we obtain a map $\phi_P \colon G \to \mathbb{Z} \ltimes \mathcal{A}(K)/\widetilde{P}$.  This map determines a unique homotopy class of maps $\phi_P \colon M_K \rightarrow B(\mathbb{Z} \ltimes \mathcal{A}(K)/\widetilde{P})$.

\begin{definition}\label{definition:obstruction}
Let $P \subset \AK$ be a lagrangian and $\widetilde{P} = \ker(\mathcal{A}(K) \rightarrow \AK \rightarrow \AK/P)$. Then we define $\psi(K,P)$ to be \[{(\phi_P)}_*([M_K]) \in H_3(B(\mathbb{Z} \ltimes \mathcal{A}(K)/\widetilde{P});\mathbb{Z})/\Inn_{\mathcal{A}(K)/\widetilde{P}}(\Z \ltimes \mathcal{A}(K)/\widetilde{P}).\]
\end{definition}

\begin{remark}
Since the inner automorphisms act on third homology by an automorphism, they preserve the property of being zero and of being nonzero.  Moreover we will always consider elements of $H_3(B(\mathbb{Z} \ltimes \mathcal{A}(K)/\widetilde{P});\mathbb{Z})$ arising from the inclusion induced map $H_3(B(\mathcal{A}(K)/\widetilde{P});\mathbb{Z}) \to H_3(B(\mathbb{Z} \ltimes \mathcal{A}(K)/\widetilde{P});\mathbb{Z})$, and on such elements the inner automorphisms act by the identity by Lemma~\ref{lemma:inner-automorphisms}.  For these two reasons we will not consider the action in the sequel, but note that in general an invariant $\psi(K,P)$ that depends only on the knot~$K$ and a choice of lagrangian $P$ lives in the quotient of the third homology by the given automorphism action $H_3(B(\mathbb{Z} \ltimes \mathcal{A}(K)/\widetilde{P});\mathbb{Z})/\Inn_{\mathcal{A}(K)/\widetilde{P}}(\Z \ltimes \mathcal{A}(K)/\widetilde{P})$.
\end{remark}

In the case that $P=0$, the invariant $\psi(K,P)$ coincides with the invariant $\beta_1$ defined in \cite[Section~$10$]{Coc04}. The obstruction $\beta_1$ was used to show that there exist distinct knots with isometric Blanchfield forms (see \cite[Theorem $10.3$]{Coc04}).

Next we show that $\psi(K,P)$ gives an obstruction for a knot to be homology ribbon $(1)$-solvable.

\begin{theorem}\label{theorem:obstruction}
Suppose that $K$ is a homology ribbon $(1)$-solvable knot via a homology ribbon $(1)$-solution $W$, and let $P \subset \AK$ be the lagrangian associated to $W$. Then $\psi(K,P)=0$.
\end{theorem}

\begin{proof}
We have $\partial W  = M_K$ and the following commutative diagram
\[\xymatrix{
\mathcal{A}(K) \ar@{->>}[r]^-{j} \ar@{^{(}->}[d]^-{i_K}
&\mathcal{A}(W)/T \ar@{^{(}->}[d]^-{i_{W}}\\
\AK \ar@{->>}[r]^-{j^{\Q}} &\mathcal{A}^{\Q}(W)
}\]
where the vertical maps are injective and the horizontal maps are surjective.
Here $\mathcal{A}(W) := H_1(W;\Z[\Z])$ and $\mathcal{A}^{\Q}(W) := H_1(W;\Q[\Z])$ are the Alexander module of $W$ and the rational Alexander module of $W$ respectively, $T$ is the $\mathbb{Z}$-torsion submodule of $\mathcal{A}(W)$, the maps labelled $i$ are the natural inclusions into the corresponding modules tensored up with $\Q$, and $j$ and $j^{\Q}$ are the maps induced by the inclusion $M_K \to W$. Since $W$ is a homology ribbon $(1)$-solution, $j$ and $j^{\Q}$ are surjections. Since $P$ is the lagrangian associated to $W$, we have that $P=\ker(j^{\Q})$, and so $\mathcal{A}^{\Q}(W) = \mathcal{A}^{\Q}(K)/P$.

Note that a splitting $\theta \colon \Z \to \pi_1(M_K)/\pi_1(M_K)^{(2)}$, together with the fact that the inclusion induced map $\Z \cong H_1(M_K;\Z) \to H_1(W;\Z) \cong \Z$ is an isomorphism, determines a splitting $\Z \to \pi_1(W)/\pi_1(W)^{(2)}$.  Use this to obtain the identification $\pi_1(W) / \pi_1(W)^{(2)} \toiso \Z \ltimes \mathcal{A}(W)$, that we use  below.

It follows from the commutative diagram above that $\ker(j) \subseteq \widetilde{P}$.  Conversely, if $x \in \widetilde{P}$, then $j^{\Q} \circ i_K (x) = 0$, and hence $x$ lies in $\ker(j)$. Hence $\mathcal{A}(K)/\widetilde{P} \cong \mathcal{A}(W)/T$. Use the inverse of this isomorphism to obtain the following commutative diagram, extending $\phi_P \colon G \rightarrow \mathbb{Z} \ltimes \mathcal{A}(K)/\widetilde{P}$.
\[
\xymatrix{
\pi_1(W) \ar[r]
&\mathbb{Z} \ltimes \mathcal{A}(W)/T \ar[rd]^-{\cong} \\
G = \pi_1(M_K) \ar[r] \ar[u] & G/G^{(2)} \cong \mathbb{Z} \ltimes \mathcal{A}(K) \ar[r] \ar[u]_{(\Id,j)} & \mathbb{Z} \ltimes \mathcal{A}(K)/\widetilde{P}.
}\]
This determines a homotopy class of maps $W \rightarrow B(\mathbb{Z} \ltimes \mathcal{A}(K)/\widetilde{P})$, and the image of the relative fundamental class $[W, M_K]$ is a $4$-chain that exhibits the vanishing \[\psi(K,P)={(\phi_P)}_*([M_K])=0 \in H_3(B(\mathbb{Z} \ltimes \mathcal{A}(K)/\widetilde{P});\mathbb{Z}).\]
\end{proof}

\begin{remark}\label{remark:obstruction}
Note that it was crucial that the $(1)$-solution $W$ be a homology ribbon $(1)$-solution, since in the above proof we made use of the fact that $j \colon \mathcal{A}(K) \to \mathcal{A}(W)/T$ is a surjective map. We do not know whether the invariant $\psi(K,P)$ has to vanish if $P$ is the lagrangian associated to a slice disc but not to any homotopy ribbon disc, or even a $(1)$-solution that is not homology ribbon.

Part of our contribution in defining this invariant carefully is to go backwards and forwards between the rational and integral Alexander modules.  The lagrangians should be indexed rationally, since they are easier to control that way, but the invariant should be integral, otherwise it would be rarely nonvanishing, as we will see.
\end{remark}

Combine Theorem~\ref{theorem:obstruction} with Lemma~\ref{lemma:homotopyribbonsol}, Lemma~\ref{lemma:kim06} and then with Lemma~\ref{lemma:kim06half}, to obtain the following obstruction theorems, which imply Theorems~\ref{thm-intro-homotopy-ribbon} and \ref{thm-intro-doubly-slice} from the introduction.

\begin{theorem}\label{ribbon-obstruction-theorem}\label{theorem:doubly-slice-obstructionhalf}
Suppose that a knot $K$ lies in $\mathcal{F}_{0.5,1}$ $($for example if $K$ is homotopy ribbon$)$. Then there is a lagrangian $P$ for the rational Blanchfield form such that $\psi(K,P)=0$.
\end{theorem}

\begin{theorem}\label{theorem:doubly-slice-obstruction}
Suppose that a knot $K$ lies in $\mathcal{F}_{1,1}$ $($for example if $K$ is doubly slice$)$.  Then there are lagrangians $P_1$ and $P_2$ for the rational Blanchfield form such that $P_1 \oplus P_2 = H_1(M_K;\Q[\Z])$ and $\psi(K,P_1) = \psi(K,P_2) = 0$.
\end{theorem}


\subsection{Computation of the invariant \texorpdfstring{$\psi(K,P)$}{psi(K,P)}}

Next we develop techniques to compute $\psi(K,P)$ in examples.
Proposition~\ref{proposition:generators}
 and the map $f_J \colon M_J \rightarrow B(\mathcal{A}(K)/\widetilde{P})$ that was defined at the end of Section~\ref{Derivatives of Knots}, combined with the canonical injection $i \colon \mathcal{A}(K)/\widetilde{P} \to \mathbb{Z} \ltimes \mathcal{A}(K)/\widetilde{P}$, appear in the following proposition.

\begin{proposition}\label{prop:psi-K-equals-H-3-of-derivative}
Suppose that $\Sigma$ is a genus $g$ Seifert surface for $K$ and that $P \subset \AK$ is a lagrangian with respect to $\Bl$. If $H$ is a metaboliser representing~$P$ and~$J$ is a derivative of~$K$ associated with~$H$, then
$$\psi(K,P) = i_*({(f_J)}_*([M_J])) \in H_3(B(\mathbb{Z} \ltimes \mathcal{A}(K)/\widetilde{P});\mathbb{Z}),$$
where $i \colon B(\mathcal{A}(K)/\widetilde{P}) \rightarrow B(\mathbb{Z} \ltimes \mathcal{A}(K)/\widetilde{P}).$
\end{proposition}

\begin{proof}
A cobordism $E$ between $M_K$ and $M_J$ was constructed in \cite{CHL10}. Here is a description of the construction of $E$, for the convenience of the reader.  Let $C$ denote the $4$-manifold obtained from $M_K \times I$ by adding $2$ handles along $J = J_1\cup \dots \cup J_g$, with zero-framing with respect to $S^3$. This is a cobordism from $M_K$ to a $3$-manifold $\partial^+ C$.  Since the components of $J$ are pairwise disjoint and form half of a symplectic basis for $H_1(\Sigma;\Z)$, the Seifert surface $\Sigma$ can be surgered to a disc inside $\partial^{+}C$. Take the union of this disc with the surgery disc in $M_K$ bounded by a zero-framed longitude of $K$, to obtain an embedded $2$-sphere $S$ in $\partial^{+}C$. Then define $E$ to be the $4$-manifold obtained from $C$ by attaching a $3$-handle along $S$. Here are some important properties of $E$, which we will call the \emph{fundamental cobordism}.

\begin{lemma}\cite[Proposition $8.1$]{CHL10} \label{lemma:properties-of-E}
The fundamental cobordism $E$ constructed above has the following properties.
\begin{enumerate}[(i)]
{\setlength\itemindent{15pt}\item\label{item:property-E-1} The map $i_* \colon \pi_1(M_K) \rightarrow \pi_1(E)$ is surjective, with the kernel the normal closure of the set of loops represented by the components of $J$.}
{\setlength\itemindent{15pt}\item\label{item:property-E-2} The meridian of the band on which $J_i \hookrightarrow \Sigma \hookrightarrow M_K =\partial^-{E}$ lies is isotopic in~$E$ to a meridian of $J_i$ in $M_J=\partial^+{E}$.}
\end{enumerate}
\end{lemma}

Now we continue with the proof of Proposition~\ref{prop:psi-K-equals-H-3-of-derivative}.
Since $\{ [J_1], \dots, [J_g] \}$ in $\AK$ spans~$P$ by Proposition~\ref{proposition:generators}~(\ref{item:prop-generators-1}), the normal closure in $\pi_1(M_K)$ of the set of loops represented by the components of $J$ is contained in the kernel of $\phi_P$. Then, by Lemma~\ref{lemma:properties-of-E}~(\ref{item:property-E-1}), we can extend $\phi_P$ uniquely over $\pi_1(E)$; denote the extension of $\phi_P$ by $\phi_E \colon \pi_1(E) \to \Z \ltimes \mathcal{A}(K)/\wt{P}$. Then $\phi_E$ induces a map
$\phi_{J} \colon \pi_1(M_J) \rightarrow \mathbb{Z} \ltimes \mathcal{A}(K)/\widetilde{P}.$
We need to show that $\phi_{J}$ agrees with
$$i \circ f_{J} \colon \pi_1(M_J) \xrightarrow{} \mathcal{A}(K)/\widetilde{P} \xrightarrow{} \mathbb{Z} \ltimes \mathcal{A}(K)/\widetilde{P},$$
where $i$ is the canonical injection and $f_J$ is the map on fundamental groups determined by the map $f_J \colon M_J \rightarrow B(\mathcal{A}(K)/\widetilde{P})$ defined at the end of the Section~\ref{Derivatives of Knots}.
By Lemma~\ref{lemma:properties-of-E}~(\ref{item:property-E-2}), for $i =1, \dots, g$ the homomorphism $\phi_J\colon \pi_1(M_J) \rightarrow \mathbb{Z} \ltimes \mathcal{A}(K)/\widetilde{P}$ sends a meridian of $J_i$ to the image of $\beta_i \in H_1(S^3 \sm \nu\Sigma;\mathbb{Z})$ under the map $H_1(S^3\sm \nu \Sigma) \rightarrow \mathcal{A}(K) \rightarrow \mathcal{A}(K)/\widetilde{P}$ (see end of the Section~\ref{Derivatives of Knots} for the definition of the $\beta_i$). Furthermore, since the images of $i \circ f_J$ and $\phi_{J}$ in $\Z \ltimes \mathcal{A}(K)/\widetilde{P}$ are abelian subgroups, they are determined by the images of the meridians of~$J$. By the definition of $f_J$, this map sends a meridian of $J_i$ to the image of $\beta_i$, for $i = 1, \dots, g$, hence $\phi_{J}$ agrees with $i \circ f_{J}$.
\end{proof}

There is a exact sequence of groups
\[0 \rightarrow \mathcal{A}(K)/\widetilde{P} \rightarrow \mathbb{Z} \ltimes \mathcal{A}(K)/\widetilde{P} \rightarrow \mathbb{Z} \rightarrow 0,\]
 that induces a fibration $B(\mathcal{A}(K)/\widetilde{P}) \rightarrow B(\mathbb{Z} \ltimes \mathcal{A}(K)/\widetilde{P}) \rightarrow S^1$. Hence we can consider the Wang exact sequence~\cite{Wan49,Mil68}:
$$\dots \rightarrow H_3(\mathcal{A}(K)/\widetilde{P}) \xrightarrow{t_*-\id} H_3(\mathcal{A}(K)/\widetilde{P}) \xrightarrow{i_*} H_3(\mathbb{Z} \ltimes \mathcal{A}(K)/\widetilde{P}) \rightarrow H_2(\mathcal{A}(K)/\widetilde{P}) \rightarrow \dots$$
where $i_*$ is induced from the inclusion and $t_*$ is induced from the action of the generator $t$ of $\mathbb{Z}$ on $B(\mathcal{A}(K)/\widetilde{P})$. By Proposition~\ref{prop:psi-K-equals-H-3-of-derivative}, $\psi(K,P)$ is the image of $[M_J] \in H_3(M_J;\Z)$ under $i_* \circ {(f_J)}_* \colon H_3(M_J;\Z) \to H_3(\Z \ltimes \mathcal{A}(K)/\wt{P})$. Furthermore, note that image of $f_J \colon \pi_1(M_J) \rightarrow \mathcal{A}(K)/\widetilde{P}$ lies in the potentially smaller subgroup
  $\mathcal{B}=\langle\beta_1, \dots, \beta_g\rangle$, where we consider $\beta_i$ as element of $\mathcal{A}(K)$ for $i \in \{1, \dots, g\}$. Hence we have actually shown that $\psi(K,P)$ is the image of $[M_J] \in H_3(M_J;\Z)$ under $i_* \circ j_* \circ (\bar{f_J})_*$, where $j_*$ and $\bar{f}_{J*}$ are from the following diagram, with $j_* \circ ({\bar{f}}_J)_* = (f_J)_*$.
\[\xymatrix{
 H_3(M_J)\ar[r]^-{({\bar{f}}_{J})_*} \ar[dr]^{(f_J)_*}   & H_3(\mathcal{B}) \ar[d]^{j_*} & \\
H_3(\mathcal{A}(K)/\widetilde{P}) \ar[r]_{t_*-\id} & H_3(\mathcal{A}(K)/\widetilde{P}) \ar[r]_-{i_*} & H_3(\mathbb{Z} \ltimes \mathcal{A}(K)/\widetilde{P}).
}\]
In particular, note that $\psi(K,P) =0$ if $(f_J)_*([M_J]) \in \im(t_*-\id)$.

\section{The relationship between \texorpdfstring{$\psi(K,P)$}{psi(K,P)} and triple linking numbers}\label{Relationship}

As above, let $K$ be a knot with a genus $g$ Seifert surface $\Sigma$, let $P \subset H_1(M_K;\Q[t,t^{-1}])=\AK$ be a lagrangian of the rational Blanchfield form, and let $J$ be a derivative of $K$ that determines a basis for a metaboliser of the Seifert form representing $P$.
Recall that by definition $\mathcal{B} = \langle\beta_1, \dots, \beta_g\rangle$ is the image of the map $f_J \colon \pi_1(M_J) \rightarrow \mathcal{A}(K)/\widetilde{P}$ defined at the end of the Section~\ref{Derivatives of Knots}.


In this section we will relate $\psi(K,P)$ to Milnor's triple linking number of the link $J$. First we prove that $\mathcal{B}$ is a finitely generated torsion-free abelian subgroup of $\mathcal{A}(K)/\widetilde{P}$.

\begin{proposition}\label{proposition:B}
Let $g$ be the genus of the Seifert surface $\Sigma$. Then $\mathcal{A}(K)/\widetilde{P}$ is a $\mathbb{Z}$-torsion free abelian group and $\mathcal{B}=\langle\beta_1, \ldots, \beta_g\rangle$ is a free abelian subgroup of $\mathcal{A}(K)/\widetilde{P}$. Moreover, if $\deg \Delta_K(t) = 2g$, then $\mathcal{B}$ has rank $g$.
\end{proposition}

\begin{proof}
First, we show that $\mathcal{A}(K)/\widetilde{P}$ is a $\mathbb{Z}$-torsion free abelian group. This is an algebraic fact.  Suppose $x\in \mathcal{A}(K)$ and $n\cdot x \in \widetilde{P}$ for some positive integer $n$, then $\iota(n\cdot x) \in P$ where $\iota \colon \mathcal{A}(K) \hookrightarrow \mathbb{Q} \otimes \mathcal{A}(K)$. Since $P$ is a $\mathbb{Q}$-vector space, we see that $\iota(x) \in P$. This implies that $x\in \widetilde{P}$, hence  $\mathcal{A}(K)/\widetilde{P}$ is $\mathbb{Z}$-torsion free.
 Then the finitely generated abelian subgroup $\mathcal{B}$ is also torsion free, hence free.

Furthermore, $\AK$ has rank $2g$ as $\mathbb{Q}$-vector space, and the lagrangian $P$ has rank $g$ as $\mathbb{Q}$-vector space. Then by Proposition~\ref{proposition:generators}~(\ref{item:prop-generators-2}), the rank of $\mathcal{B}$ is at least $g$.  Since $\mathcal{B}$ is generated by $g$ elements, we conclude that $\mathcal{B}$ has rank $g$.
\end{proof}

Recall that, for a 3-component link $L = L_1\cup L_2 \cup L_3$ with zero pairwise linking numbers, Milnor's triple linking number $\bar{\mu}_L(123) \in \Z$ is defined~\cite{Mi54,Mi57}. Also, recall that $\bar{\mu}_L(123)$ can be calculated as signed count of triple intersection points of Seifert surfaces~\cite[Section~5]{Co85}.

Consider a map $f_L \colon M_L \rightarrow S^1 \times S^1 \times S^1$ such that the induced map on first homology followed by the canonical isomorphism $H_1((S^1)^3;\Z) \cong \Z^3$ sends the first meridian to $(1,0,0)$, the second meridian to $(0,1,0)$ and the third meridian to $(0,0,1)$. For $i=1,2,3$, let $f_i \colon M_L \rightarrow S^1$ be the map obtained from $f_L$ by projecting onto $i$th factor. We claim that for $i=1,2,3$, we can alter $f_i$ by a homotopy so that ${f_i}^{-1}(1)$ is a capped off Seifert surface for the $i$th component of $L$.
To see this, we argue as follows.  Given a capped-off Seifert surface $F_i$ for $L_i$, the Pontryagin-Thom construction gives rise to a map $\wt{f}_i\colon M_L \to S^1$ with $\wt{f}_i^{-1}(\{1\}) = F_i$.  Homotopy classes of maps to $S^1$ correspond to elements of $H^1(M_L;\Z)$.  Since the cohomology classes of $f_i$ and $\wt{f}_i$ are equal, they are homotopic maps.
So indeed $f_i$ is homotopic to a map such that $1 \in S^1$ is a regular point and the inverse image of $1$ is a capped-off Seifert surface for $L_i$.

After the alterations, $f_1 \times f_2 \times f_3 \colon M_L \to (S^1)^3$ is still homotopic to $f$ and further the signed count of $(f_1 \times f_2 \times f_3)^{-1}(\{1\} \times \{1\} \times \{1\})$ coincides with signed count of the number of triple intersection points of Seifert surfaces. Finally, since the signed count of $(f_1 \times f_2 \times f_3)^{-1}(\{1\} \times \{1\} \times \{1\})$ is also equal to the degree of $f_1 \times f_2 \times f_3$, we can conclude that $f_*([M_L]) = \bar{\mu}_L(123) \in H_3(S^1 \times S^1 \times S^1) = \mathbb{Z}$. Using this observation we will be able to relate $\psi(K,P)$ with Milnor's triple linking number of $J$.

We restrict our attention to the case that $\deg \Delta_K(t) = 2g$ where $g$ is the genus of the Seifert surface.
 From $f_J \colon \pi_1(M_J) \rightarrow \mathcal{A}(K)/\widetilde{P}$, we obtain $\bar{f}_J \colon \pi_1(M_J) \rightarrow \mathcal{B} = \mathbb{Z}^g$ and $j \colon \mathcal{B} = \mathbb{Z}^g \rightarrow \mathcal{A}(K)/\widetilde{P}$ where $j$ is an inclusion and $f_J = j \circ \bar{f}_J$. Then $\bar{f}_J$ induces a map $\bar{f}_J \colon M_J \rightarrow \prod^{g} S^1$ which sends the $i$-th meridian to the class $e_i = (0,\dots,1,\dots,0)$ of the first homology, for $i\in \{1,\dots,g\}$ as above. Let $\{e_i \times e_j \times e_k \mid 1 \leq i < j < k \leq g \}$ be a basis for $H_3(\prod^{g}S^1)$, where $\times$ is the homology product. We have shown the following proposition.

\begin{proposition}\label{proposition:milnor-and-obstruction}
Suppose that $\deg\Delta_K(t) = 2g$, where $g$ is the genus of a Seifert surface $\Sigma$ for $K$, and let $J = J_1\cup \dots \cup J_g$ be a derivative of $K$ on $\Sigma$. Then $\bar{f}_{J*}([M_J]) \in H_3(\mathcal{B}) \cong H_3(\mathbb{Z}^g) \cong H_3(\prod^{g} S^1)$ has coordinates $\{\bar{\mu}_{J}(ijk)\}$ with respect to the basis $\{e_i \times e_j \times e_k \mid 1 \leq i < j < k \leq g \}$.
\end{proposition}

\noindent Combining Proposition~\ref{prop:psi-K-equals-H-3-of-derivative} and Proposition~\ref{proposition:milnor-and-obstruction} gives rise to the following theorem.

\begin{theorem}\label{theorem:milnor-and-obstruction} Suppose that $\deg\Delta_K(t) = 2g$, where $g$ is the genus of a Seifert surface $\Sigma$ for $K$ and suppose $J = J_1\cup \dots \cup J_g$ is a derivative on $\Sigma$ associated with a metaboliser $H$. Then the invariant $\psi(K,P)$ is the image of $\sum_{i<j<k}\, \bar{\mu}_{J}(ijk)[e_i \times e_j \times e_k]$, under the dashed map in the diagram below, where $P$ is the lagrangian represented by $H$.

\[
\xymatrix{
& H_3(\mathcal{B})\ar[d]^-{j_*} \ar @{-->}[rd]
& \\
H_3(\mathcal{A}(K)/\widetilde{P}) \ar[r]^{t_*-\Id} & H_3(\mathcal{A}(K)/\widetilde{P}) \ar[r]^{i_*} & H_3(\mathbb{Z} \ltimes \mathcal{A}(K)/\widetilde{P}).
}\]
\end{theorem}

More generally, we have a sufficient condition for $\psi(K,P)$ to vanish. In Section~\ref{section:possiblemilnor}, we will present an equivalent condition for $\psi(K,P)$ to vanish for some special cases.
%
%
%
\begin{theorem}\label{thm:vanishesiftriple}
Suppose that $K$ has a derivative $J$ associated with a metabolizer $H$ such that all triple linking numbers of $J$ vanish. Then $\psi(K,P)$ vanishes, where $P$ is the lagrangian represented by $H$.
\end{theorem}

\begin{proof}
Let $\bar{f}_J \colon \pi_1(M_J) \rightarrow \mathcal{B}$ be the map defined above, where $\mathcal{B}$ is a free abelian group by Proposition~\ref{proposition:B}. Let $\ab_J\colon \pi_1(M_J) \rightarrow H_1(M_J)$ be the abelianization map and let $\pr_J\colon H_1(M_J) \rightarrow \mathcal{B}$ be the projection map, so that $\pr_J \circ \ab_J = \bar{f}_J$. As above $\ab_J$ induces a map $\ab_J\colon M_J \rightarrow \prod^{g} S^1$, where $g$ is the rank of $H_1(M_J)$. Let $\{e_i \times e_j \times e_k \mid 1 \leq i < j < k \leq g \}$ be a basis for $H_3\big(\prod^{g}S^1\big)$. Then $(\ab_J)_{*}([M_J]) \in H_3(H_1(M_J))$ has coordinates $\{\bar{\mu}_{J}(ijk)\}$ with respect to the basis $\{e_i \times e_j \times e_k \mid 1 \leq i < j < k \leq g \}$. Since we are assuming that $J$ has all triple linking number vanishing, $(\ab_J)_{*}([M_J]) \in H_3(H_1(M_J))$ vanishes. This concludes the proof, since \[\psi(K,P) =i_* \circ j_* \circ (\bar{f}_J)_{*}([M_J])= i_* \circ j_* \circ (\pr_J)_{*} \circ (\ab_J)_{*}([M_J]).\] \end{proof}

\section{Determining the possible Milnor's invariants of derivatives}\label{section:possiblemilnor}

In this section we consider an algebraically slice knot $K$ with a genus three Seifert surface $\Sigma$. For the rest of this section, fix the following notation.  Let $H \subset H_1(\Sigma;\Z)$ be a metaboliser of the Seifert form of $K$ with respect to $\Sigma$, let $J = J_1\cup J_2 \cup J_3$ be a derivative of $K$ associated with $H$ and let $\delta_1 \cup \delta_2 \cup \delta_3$ be intersection duals of $J = J_1\cup J_2 \cup J_3$ on~$\Sigma$. Let $X := \big(\lk(\delta_i,J^+_j)\big)_{3 \times 3}$ be the linking matrix of the $\delta_i$ and the $J_j$.  Here $J^+_j$ is a positive push-off of $J_j$.
 With respect to the basis $\{\delta_1,\delta_2,\delta_3,J_1,J_2,J_3\}$ of $H_1(\Sigma;\Z)$, the Seifert form is of the type
 \[\bp A & X \\ X - \Id & 0 \ep.\]
 Recall that we denoted the set of the derivatives on $\Sigma$ associated with $H$ by $\mathfrak{d} K /\mathfrak{d}H$. As in the introduction, define
 \[S_{K,H} :=\{ \bar{\mu}_{L}(123) - \bar{\mu}_{L'}(123) \mid L,L' \in \mathfrak{d} K /\mathfrak{d}H, \mathfrak{o}(L) = \mathfrak{o}(L')\}\]
where $\mathfrak{o}(L) := [L_1] \wedge [L_2] \wedge [L_3] \in \bigwedge^3 H$.
The following result was proven by the first author in~\cite{Park16}.

\begin{theorem}\label{theorem:inclusion}
$S_{K,H} \supseteq \big(\det(X) - \det(X-\Id)\big) \mathbb{Z}$.
\end{theorem}

The proof of  Theorem~\ref{theorem:inclusion} used a geometric argument to show how to change one derivative to another, changing the triple linking number by $\det(X) - \det(X-\Id)$.
In this section we apply Theorem~\ref{theorem:milnor-and-obstruction} to obtain inclusions in the opposite direction, namely limitations on the possible changes of $\ol{\mu}$-invariants.  We will show that the inclusion in Theorem~\ref{theorem:inclusion} is an equality in some special cases.  We do not know whether it is an equality in general.

Let $n,d$ be integers and write $g_{(n,d,i)} = \gcd(n,d^i)$ for a positive integer $i$. Note that $g_{(n,d,i)}$ stabilises to some integer as $i$ gets large; we will denote this integer by $g_{(n,d)}$.

\begin{theorem}\label{theorem:otherinclusion}
In the notation introduced at the start of Section~\ref{section:possiblemilnor}, suppose that $$X := \big(\lk(\delta_i,J^+_j)\big)_{3\times3}$$ is a diagonal matrix $\diag(p_1,p_2,p_3)$ such that $p_i\cdot(p_i-1) \neq 0$ for each $i \in \{1,2,3\}$.
Let $n :=\det(X) - \det(X-\Id)$. Then $S_{K,H} \subseteq \frac{n}{m}\mathbb{Z}$, where $$m = \lcm(g_{(n,p_1)}, g_{(n,p_2)}, g_{(n,p_3)}, g_{(n,p_1-1)}, g_{(n,p_2-1)}, g_{(n,p_3-1)}).$$
\end{theorem}

\noindent Note that it is automatic from the definitions of $n$ and $m$ that $m$ divides $n$.

\begin{proof}
Let $P \subset \AK$ be a lagrangian represented by $H$. Write $\Lambda = \Z[t,t^{-1}]$.  We have the following computations:
\begin{equation}\label{equation:alexandermodule}
\begin{split}
   \mathcal{A}(K)/\widetilde{P} & \cong \Lambda^3/\langle X^T - t X^T - \Id\rangle\\
   & \cong \Lambda^3 / \langle(1-t)X^T - \Id\rangle\\
   & \cong \Lambda/\langle(p_1-1)-p_1t\rangle \oplus \Lambda/\langle(p_2-1)-p_2t\rangle \oplus\Lambda/\langle(p_3-1)-p_3t\rangle
\end{split}
\end{equation}
 Also, note that $\deg \Delta_K(t) = 2g = 6$.

Fix a generator $\mathfrak{o}_{\fix}$ of $\bigwedge^3 H$. Let $J$ and $J'$ be two derivatives of $K$ associated with $H$ such that $\mathfrak{o}(J)= \mathfrak{o}(J') = \mathfrak{o}_{\fix}$. By Theorem~\ref{theorem:milnor-and-obstruction}, it follows that \[\psi(K,P) = i_* \circ j_* (\bar{\mu}_{J}(123) \cdot e_1\times e_2 \times e_3) = i_* \circ j_* (\bar{\mu}_{J'}(123)\cdot e_1\times e_2 \times e_3),\] where $e_1\times e_2 \times e_3$ is a generator for $H_3(\mathcal{B}) \cong H_3(\mathbb{Z}^3) \cong \mathbb{Z}$ corresponding to $\mathfrak{o}_{\fix}$.
Hence 
\[j_* ((\bar{\mu}_{J}(123)-\bar{\mu}_{J'}(123))\cdot e_1\times e_2 \times e_3) \in \ker(i_*) = \im (t_* - \id).\]
Moreover, note that since $\mathcal{A}(K)/\widetilde{P}$ is a $\mathbb{Z}$-torsion free from Proposition~\ref{proposition:B}, we have an isomorphism  of $\Lambda$-modules $\wedge^3(\mathcal{A}(K)/\widetilde{P}) \cong H_3(\mathcal{A}(K)/\widetilde{P})$ \cite[Chapter $5$]{Bro94}. By~(\ref{equation:alexandermodule}) we see that
\[H_3(\mathcal{A}(K)/\widetilde{P})\cong \Lambda/\langle(p_1-1)-p_1t\rangle \otimes \Lambda/\langle(p_2-1)-p_2t\rangle \otimes \Lambda/\langle(p_3-1)-p_3t\rangle.\] In particular $H_3(\mathcal{A}(K)/\widetilde{P})$ is a $\mathbb{Z}$-torsion free module and $j_*(e_1\times e_2 \times e_3) = 1 \otimes 1 \otimes 1 \in H_3(\mathcal{A}(K)/\widetilde{P})$ is nonzero element, hence $j_*$ is an injective map. Therefore, it will be enough to show that $\im(t_*-\id) \cap \im(j_*) \subseteq \langle\smfrac{n}{m} \cdot j_*(e_1\times e_2 \times e_3)\rangle$ to get our desired result.
Consider the map \[\ell \colon H_3(\mathcal{A}(K)/\widetilde{P}) \rightarrow H_3(\mathcal{A}(K)/\widetilde{P}) \otimes \mathbb{Q} \cong \mathbb{Q}.\] Here the isomorphism to $\Q$ follows since $\Lambda/\langle(p_i-1)-p_i t\rangle \otimes \Q \cong \Q$ for $i=1,2,3$. Then note that the image of $\ell \circ j_*$ is contained in $\Z \subset \Q$. Let $f_1(t) \otimes f_2(t) \otimes f_3(t)$ be any element in $H_3(\mathcal{A}(K)/\widetilde{P})$, and suppose that $(t_* -\id)\big(f_1(t) \otimes f_2(t) \otimes f_3(t)\big) \in \im(j_*)$. We calculate:
\begin{equation}\label{equation:image}
\begin{split}
& \ell\circ (t_*-\id)\big(f_1(t) \otimes f_2(t) \otimes f_3(t)\big)\\
    =& \ell \big(tf_1(t) \otimes tf_2(t) \otimes tf_3(t) - f_1(t) \otimes f_2(t) \otimes f_3(t)\big)\\
    =& \frac{(p_1-1)(p_2-1)(p_3-1)-p_1p_2p_3}{p_1p_2p_3} \left(f_1\Big(\smfrac{p_1-1}{p_1}\Big) \otimes f_2\Big(\smfrac{p_2-1}{p_2}\Big) \otimes f_3\Big(\smfrac{p_3-1}{p_3}\Big)\right)\\
    =& n\,\frac{-f_1\Big(\smfrac{p_1-1}{p_1}\Big) \otimes f_2\Big(\smfrac{p_2-1}{p_2}\Big) \otimes f_3\Big(\smfrac{p_3-1}{p_3}\Big)}{p_1p_2p_3}.
\end{split}
\end{equation}

Note from~(\ref{equation:image}) that the factors of the denominator of $\ell\circ (t_*-\id)\big(f_1(t) \otimes f_2(t) \otimes f_3(t)\big)$ come from the list $p_1,p_2,p_3,p_1-1,p_2-1,p_3-1$.  Hence if $$\ell\circ (t_*-\id)\big(f_1(t) \otimes f_2(t) \otimes f_3(t)\big) \in \mathbb{Z},$$ then $\ell\circ (t_*-\id)\big(f_1(t) \otimes f_2(t) \otimes f_3(t)\big)$ is divisible by $\frac{n}{m}$. For any element $N\cdot(1\otimes1\otimes1) \in \im(t_*-\id) \cap \im(j_*)$, where $N \in \Z$, we see that $\frac{n}{m}$ divides $N$, which concludes the proof of Theorem~\ref{theorem:otherinclusion}.
\end{proof}

Now we can prove Theorem~\ref{thm:determination-of-mu-123}, in which we determine the set $S_{K,H}$ precisely for certain knots and certain genus 3 Seifert surfaces.

\begin{theoremB}\label{corollary:otherinclusion}
In the notation introduced at the start of Section~\ref{section:possiblemilnor}, suppose that $$X := \big(\lk(\delta_i,J^+_j)\big)_{3\times3}$$ is a diagonal matrix $\diag(p_1,p_2,p_3)$.  Write $n := \det(X) - \det(X-\Id)$.  Suppose that $\gcd(p_i,n)=\gcd(p_i-1,n)=1$ and $p_i\cdot(p_i-1) \neq 0$ for all $i = 1,2,3$.
Then $$S_{K,H} = n\mathbb{Z}.$$
\end{theoremB}

\begin{proof} We have $S_{K,H} \supseteq n\mathbb{Z}$ from Theorem~\ref{theorem:inclusion}. In order to see $S_{K,H} \subseteq n\mathbb{Z}$ observe that $g_{(n,p_i)}=g_{(n,p_i-1)} = 1$, for $i = 1,2,3$ from Theorem~\ref{theorem:otherinclusion}. Hence $m=1$ in Theorem~\ref{theorem:otherinclusion} and this implies the desired result.
\end{proof}

There certainly exist integers $p_1,p_2,p_3$ that satisfy the assumptions of Theorem~\ref{thm:determination-of-mu-123} (for detailed calculations see Proposition~\ref{prop:propertyS} (\ref{propertyS:item1})). For instance, $p_1 = 3, p_2 = 5$, and $p_3 = 17$ satisfy the assumptions. We present the following corollary, which relates the homology ribbon obstruction with Milnor's triple linking number.

\begin{corollary}\label{corollary:TFAE} Suppose that
$$X := \big(\lk(\delta_i,J^+_j)\big)_{3\times3}$$ is a diagonal matrix  $\diag(p_1,p_2,p_3)$ such that $p_i\cdot(p_i-1) \neq 0$ for $i =1,2,3$.  Let $n =\det(X) - \det(X-\Id)$.
Let \[m := \lcm(g_{(n,p_1)}, g_{(n,p_2)}, g_{(n,p_3)}, g_{(n,p_1-1)}, g_{(n,p_2-1)}, g_{(n,p_3-1)})\] and let $P \subset \AK$ be a lagrangian represented by $H$. The following statements are equivalent.
\begin{enumerate}
\item\label{item:TFAE-1} For any derivative $J = J_1 \cup J_2 \cup J_3$ associated with $H$, $\bar{\mu}_{J}(123) \equiv 0 \mod{\frac{n}{m}}$.
\item\label{item:TFAE-2} There exists a derivative $J = J_1 \cup J_2 \cup J_3$ associated with $H$ such that $\bar{\mu}_{J}(123) \equiv 0 \mod {\frac{n}{m}}$.
\item\label{item:TFAE-3} $\psi(K,P) \equiv 0$.
\end{enumerate}
\end{corollary}

\begin{proof}
That (\ref{item:TFAE-1}) implies (\ref{item:TFAE-2}) is straightforward, since every metaboliser can be represented by a derivative link. To see (\ref{item:TFAE-3}) implies (\ref{item:TFAE-1}), we argue as follows. Assume that $\psi(K,P) \equiv 0$. Then for any derivative $J = J_1 \cup J_2 \cup J_3$ associated with $H$, \[j_* (\bar{\mu}_{J}(123) \cdot e_1\times e_2 \times e_3) \in \im(t_*-\id)\cap \im(j_*).\]  In addition, from the proof of Theorem~\ref{theorem:otherinclusion} we saw that $\im(t_*-\id) \cap \im(j_*) \subseteq \langle\frac{n}{m} \cdot j_*(e_1\times e_2 \times e_3)\rangle$, hence we conclude that $\frac{n}{m}$ divides $\bar{\mu}_{J}(123)$.  This completes the proof that (\ref{item:TFAE-3}) implies (\ref{item:TFAE-1}).

In order to prove that (\ref{item:TFAE-2}) implies (\ref{item:TFAE-3}), we only need to show that $\im(t_*-\id) \cap \im(j_*) = \langle\frac{n}{m} \cdot j_*(e_1\times e_2 \times e_3)\rangle$, since $\psi(K,P) = i_* \circ j_* (\bar{\mu}_{J}(123) \cdot e_1\times e_2 \times e_3)$ by Theorem~\ref{theorem:milnor-and-obstruction}. For $i=1,2,3$ there is a canonical identification
\[\Lambda/\langle(p_i-1)-p_it\rangle \cong \mathbb{Z}\Big[\smfrac{p_i-1}{p_i},\smfrac{p_i}{p_i-1}\Big] \cong \mathbb{Z}\Big[\smfrac{1}{p_i},\smfrac{1}{p_i-1}\Big].\]
  Moreover, for $i=1,2,3$ we have that $\gcd(g_{(n,p_i)},g_{(n,p_i-1)}) = 1$. Hence there exists an element $f_i(t) \in \Lambda/\langle(p_i-1)-p_it\rangle$ that corresponds to $\frac{1}{g_{(n,p_i)}\cdot g_{(n,p_i-1)}} \in \mathbb{Z}\Big[\smfrac{1}{p_i},\smfrac{1}{p_i-1}\Big]$. Let $$s=\frac{g_{(n,p_1)}g_{(n,p_2)}g_{(n,p_3)}g_{(n,p_1-1)}g_{(n,p_2-1)}g_{(n,p_3-1)}}{m} \in \mathbb{Z}$$ and consider \[s\cdot p_1 f_1(t) \otimes p_2f_2(t) \otimes p_3f_3(t) \in H_3(\mathcal{A}(K)/\widetilde{P}).\] We calculate:
\begin{equation*}
\begin{split}
& (t_*-\id)\big(s\cdot p_1f_1(t) \otimes p_2f_2(t) \otimes p_3f_3(t)\big)\\
    =& s(p_1-1)f_1(t) \otimes (p_2-1)f_2(t) \otimes (p_3-1)f_3(t) - sp_1f_1(t) \otimes p_2f_2(t) \otimes p_3f_3(t)\\
    =& s((p_1-1)(p_2-1)(p_3-1) -p_1p_2p_3)\cdot\big(f_1(t) \otimes f_2(t) \otimes f_3(t)\big)\\
    =& -ns\cdot\big(f_1(t) \otimes f_2(t) \otimes f_3(t)\big)\\
    =& -\frac{n}{m}g_{(n,p_1)}g_{(n,p_2)}g_{(n,p_3)}g_{(n,p_1-1)}g_{(n,p_2-1)}g_{(n,p_3-1)} \cdot\big(f_1(t) \otimes f_2(t) \otimes f_3(t)\big)\\
    =& -\frac{n}{m}\cdot\big(g_{(n,p_1)}g_{(n,p_1-1)}f_1(t) \otimes g_{(n,p_2)}g_{(n,p_2-1)}f_2(t) \otimes g_{(n,p_3)}g_{(n,p_3-1)}f_3(t)\big)\\
    =& -\frac{n}{m}\cdot j_*(e_1\times e_2 \times e_3).
\end{split}
\end{equation*}
Therefore $\im(t_*-\id) \cap \im(j_*) \supseteq \langle\frac{n}{m} \cdot j_*(e_1\times e_2 \times e_3)\rangle$, which concludes the proof that (\ref{item:TFAE-2}) implies (\ref{item:TFAE-3}).
\end{proof}



\section{Algebraically slice knots with non-vanishing \texorpdfstring{$\mathbb{Z}[\mathbb{Z}]$}{Z[Z]} homology ribbon obstruction}

In this section, we construct algebraically slice knots with non-vanishing $\mathbb{Z}[\mathbb{Z}]$ homology ribbon obstruction. Later, we will relate these examples to the doubly-solvable filtration and to a generalised version of the Kauffman conjecture. The following proposition is an observation from \cite[Section 9]{Lev69}. We give a quick proof for the convenience of the reader.

\begin{proposition}\label{proposition:Levine}
Let $H \subset H_1(F)$ be a metaboliser of the Seifert form and let $M$ be a Seifert matrix $($i.e.\ a square matrix over $\Z$ such that $M-M^{T}$ is invertible$)$, that is itself invertible over $\mathbb{Q}$. If $H_\mathbb{Q} = H \otimes \mathbb{Q}$, then $M^{-1}M^T (H_\mathbb{Q}) =H_\mathbb{Q}$.
\end{proposition}

\begin{proof}
Let $x \in H_\mathbb{Q}$ and let $M^{-1}M^{T}x = y$, then $M^{T}x = My$. Then for any element $z \in H_\mathbb{Q}$, $z^T M y = z^T M^{T}x = 0$, hence $y\in H_{\mathbb{Q}}^{\perp} = H_\mathbb{Q}$.
\end{proof}

We will apply Proposition~\ref{proposition:Levine} to a knot to compute all possible metabolisers when a Seifert form of the knot satisfies certain conditions.

\begin{proposition}\label{proposition:possiblemetaboliser}
Let $M$ be a $6\times 6$ Seifert matrix that is invertible over $\mathbb{Q}$. Moreover, assume that $M^{-1}M^{T}$ has six distinct eigenvalues $\lambda_1, \dots, \lambda_6$ where $\lambda_i \neq 1$ for $i \in \{1, \dots, 6 \}$ and let $v_1, \dots, v_6$ be eigenvectors associated to $\lambda_1, \dots, \lambda_6$. Then the possible metabolisers of the Seifert form $M$ are precisely the subspaces $\spn(v_i, v_j, v_k) \cap \mathbb{Z}^6$, where $v_i, v_j, v_k$ represent curves on the Seifert surface with pairwise intersection zero, that is $v_i^{T}Mv_j = v_j^{T}Mv_i$, $v_i^{T}Mv_k = v_k^{T}Mv_i$, and $v_j^{T}Mv_k = v_k^{T}M v_j$.
\end{proposition}

\begin{proof}
For simplicity, assume that $v_1, v_2$ and $v_3$ have pairwise intersection zero. We claim that $\spn(v_1, v_2, v_3) \cap \mathbb{Z}^6$ is a metaboliser. Since $M^{-1}M^{T}v_i = \lambda_i\cdot v_i$ we have $M^{T}v_i = \lambda_i\cdot Mv_i$ and $v_j^{T}M^{T}v_i = \lambda_i\cdot v_j^{T}Mv_i = \lambda_i\cdot v_i^{T}M^{T}v_j=\lambda_i\cdot v_j^{T}M^{T}v_i$ for all $i,j \in \{1,2,3\}$. Therefore $\spn(v_1, v_2, v_3) \cap \mathbb{Z}^6$ is a metaboliser as claimed.

For the other direction, let $H$ be a metaboliser and let $H_\mathbb{Q} = H \otimes \mathbb{Q}$. Let $x=a_1v_1 + \cdots + a_6v_6$ be an element of $H_\mathbb{Q}$. Then by Proposition~\ref{proposition:Levine}, $(M^{-1}M^T)^jx =  \lambda_1^ja_1v_1 + \cdots + \lambda_6^ja_6v_6 \in H_\mathbb{Q}$ for all $j$. In particular the column space of the $6\times 6$ matrix  $$E = \begin{pmatrix} \lambda_i^{j-1} a_i \end{pmatrix},$$ with respect to the basis $\{v_1,\dots,v_6\}$, is contained in $H_\mathbb{Q}$. Observe that the rank of $E$, which is the rank of the row space of $E$, is equal to the number of nonzero $a_i$. Since the dimension of $H_\mathbb{Q}$ is $3$, we conclude that the number of nonzero $a_i$ is at most $3$. Hence $H=\spn(v_i, v_j, v_k) \cap \mathbb{Z}^6$, where $v_i, v_j, v_k$ have pairwise intersection zero. \end{proof}

\noindent We have the following corollary for a rather specific case.

\begin{corollary}\label{corollary:eightmetabolisers}
In the notation introduced at the start of Section~\ref{section:possiblemilnor}, suppose that $$X := \big(\lk(\delta_i,J^+_j)\big)_{3\times3}$$ is a diagonal matrix $\diag(p_1,p_2,p_3)$ such that $p_i\cdot(p_i-1) \neq 0$ for $i\in \{1,2,3\}$ and such that \[\smfrac{p_1}{p_1-1},\smfrac{p_1-1}{p_1},\smfrac{p_2}{p_2-1},\smfrac{p_2-1}{p_2},\smfrac{p_3}{p_3-1} \text{ and } \smfrac{p_3-1}{p_3}\] are distinct rational numbers. In addition, assume that $A := \big(\lk(\delta_i,\delta^+_j)\big)_{3\times3}$ is the $3 \times 3$ zero matrix. Then~$K$ has eight possible metabolisers.
\end{corollary}

\begin{proof}
Let $M$ be the Seifert matrix with respect to the basis \[\{\alpha_1=[\delta_1], \alpha_2=[J_1], \alpha_3 = [\delta_2], \alpha_4=[J_2], \alpha_5=[\delta_3], \alpha_6=[J_3]  \}.\] Then
$$M^{-1}M^{T}=  \diag\left(\smfrac{p_1}{p_1-1},\smfrac{p_1-1}{p_1},\smfrac{p_2}{p_2-1},\smfrac{p_2-1}{p_2},\smfrac{p_3}{p_3-1},\smfrac{p_3-1}{p_3}\right)$$
is a $6 \times 6$ diagonal matrix with six distinct eigenvalues \[\lambda_1 = \smfrac{p_1}{p_1-1},\lambda_2 =\smfrac{p_1-1}{p_1},\lambda_3 =\smfrac{p_2}{p_2-1},\lambda_4 =\smfrac{p_2-1}{p_2},\lambda_5 =\smfrac{p_3}{p_3-1},\lambda_6 =\smfrac{p_3-1}{p_3}\] and $\alpha_i$ is an eigenvector associated with $\lambda_i$ for $i \in \{1, \dots, 6 \}$. Furthermore, since $\delta_1 \cup \delta_2 \cup \delta_3$ are the intersection duals of $J_1 \cup J_2 \cup J_3$, we conclude that the following comprise all the possible metabolisers:
\[H = \spn\{ \alpha_i, \alpha_j, \alpha_k \} \text{ where } i \in \{1,2\}, j \in \{3,4\}, k \in \{5,6\}. \qedhere\]
\end{proof}

Next we will present examples of knots $K$ where $\psi(K,P) \neq 0$ for all possible lagrangians $P \subset \AK$.  These examples will be used to prove Theorems \ref{theorem:doubly-slice-theorem} and \ref{thm:any-seifert-surface-non-trivial-mu-123}.

\begin{example}\label{example1}
We continue to use the notation from the start of Section~\ref{section:possiblemilnor}.  In particular, with respect to the basis $\{\delta_1,\delta_2,\delta_3,J_1,J_2,J_3\}$, the Seifert form looks like  \[\bp A & X \\ X - \Id & 0 \ep.\]
 Suppose that $$X := \big(\lk(\delta_i,J^+_j)\big)_{3\times3} = \diag(p_1,p_2,p_3),$$ and assume that $A = \big(\lk(\delta_i,\delta^+_j)\big)_{3\times3}$ is the $3 \times 3$ zero matrix. Start with the knot drawn in Figure~\ref{figure:diskband2}.  We have a disc-band form for the Seifert surface $\Sigma$, also depicted in Figure~\ref{figure:diskband2}. Perform double Borromean rings insertion moves to tie Borromean rings into the bands of the Seifert surface by string link infections (for a precise definition of string link infection see \cite[Section~2.4]{Park16}, for instance)
 to arrange that $\bar{\mu}_{L}(123) = 1$ for each of the $8$ choices of $L=L_1\cup L_2 \cup L_3$ with $L_i \in \{J_i, \delta_i\}$ for each $i =1,2,3$. Let
$$\begin{pmatrix}
n_1\\
n_2\\
n_3\\
n_4\end{pmatrix} := \begin{pmatrix}
p_1\cdot p_2 \cdot p_3 - (p_1-1)\cdot (p_2-1) \cdot (p_3-1)\\
p_1\cdot p_2 \cdot (p_3-1) - (p_1-1)\cdot (p_2-1) \cdot p_3\\
p_1\cdot (p_2-1) \cdot p_3 - (p_1-1)\cdot p_2 \cdot (p_3-1)\\
(p_1-1)\cdot p_2 \cdot p_3 - p_1\cdot (p_2-1) \cdot (p_3-1)\end{pmatrix}$$ and let $m_i = \lcm\big(g(n_i,p_1),g(n_i,p_2),g(n_i,p_3),g(n_i,p_1-1),g(n_i,p_2-1),g(n_i,p_3-1)\big)$ for $i \in \{1,2,3,4\}$.

Note that there are infinitely many triples of integers $\{p_1, p_2, p_3\}$ such that $|\frac{n_i}{m_i}|>1$ and $p_i\cdot(p_i-1) \neq 0$ for $i = 1,2,3$ (see the proof of Proposition~\ref{prop:propertyS}~(\ref{propertyS:item1})). Suppose that $\{p_1, p_2, p_3\}$ is a such triple. Then by Corollary~\ref{corollary:eightmetabolisers}, there are eight possible metabolisers:
$$H = \spn\{ [L_1], [L_2], [L_3] \} \text{ where } L_i \in \{J_i,\delta_i\} \text{ for } i\in \{1,2,3\}.$$
Let $P \subset \AK$ be some lagrangian of $K$ and note that by Lemma~\ref{lemma:Lagrangian}, $P$ can be represented by some metaboliser. Since $|\frac{n_i}{m_i}| > 1$ and $\bar{\mu}_{L}(123) = 1$, for any link $L$ with $L_i \in \{J_i, \delta_i\}$ for $i =1,2,3$, it follows from Corollary~\ref{corollary:TFAE} that $\psi(K,P) \neq 0$.

\begin{figure}[htbp]
\labellist\small
\pinlabel{$p_1$} at 97 80
\pinlabel{$p_2$} at 295 80
\pinlabel{$p_3$} at 493 80
\pinlabel{$\delta_1$} at  63 20
\pinlabel{$\delta_2$} at 260 20
\pinlabel{$\delta_3$} at 460 20
\pinlabel{$J_1$} at 127 20
\pinlabel{$J_2$} at 325 20
\pinlabel{$J_3$} at 525 20
    \endlabellist
   \includegraphics[scale=0.7]{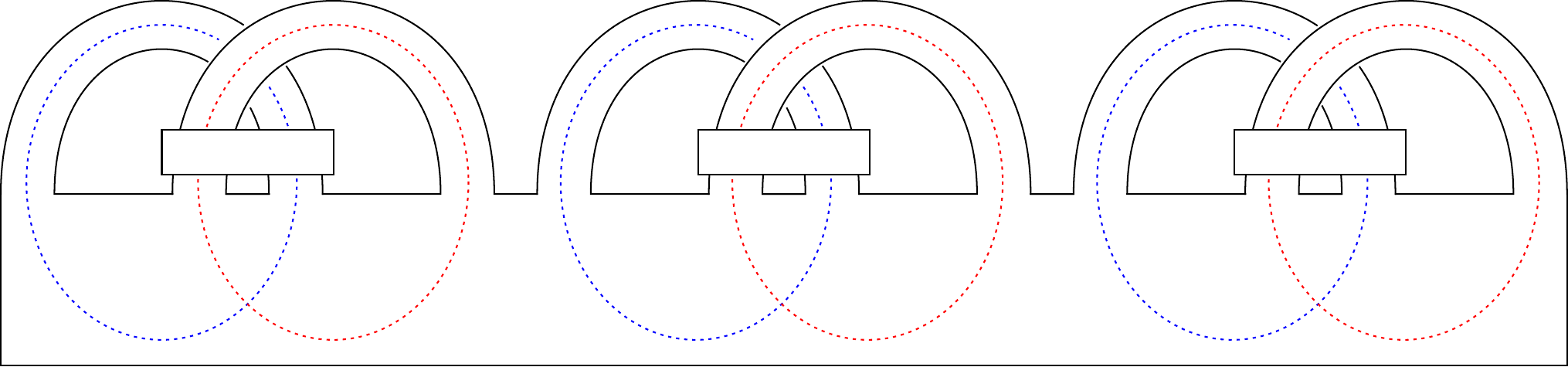}
    \caption{Disc-band form for $\Sigma$, where a solid box represents $p_i$ full twists between two bands with no twist on each of the bands. Since all triples in the figure have Milnor's triple linking number zero, to get the desired example, perform the double of a Borromean rings insertion move eight times, on the $i$th, $j$th and $k$th bands for all choices of $i \in \{1,2\}, j \in \{3,4\}$ and $k \in \{5,6\}$.}
    \label{figure:diskband2}
  \end{figure}
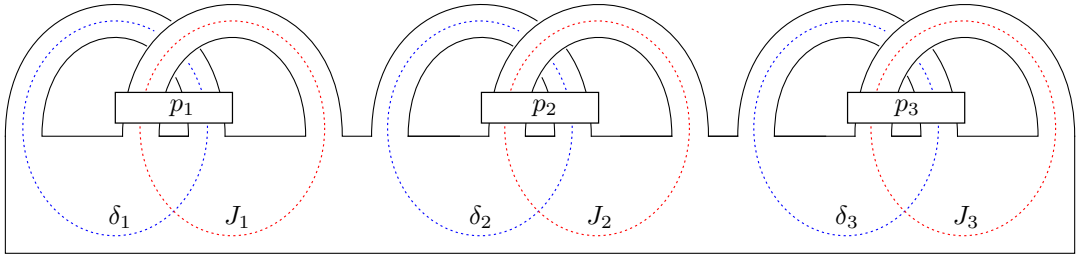
\end{example}

\begin{example}\label{example2}
Let $K$ be the knot shown in Figure~\ref{figure:example3}. Then we have $$X = \big(\lk(\delta_i,J^+_j)\big)_{3\times3} = \diag(p_1,p_2,p_3), A = \big(\lk(\delta_i,\delta^+_j)\big)_{3\times3} = \diag(1,-1,1).$$
For $i =1,2,3$, the curve $\varepsilon_i$ from Figure~\ref{figure:example3} has self linking number zero (i.e.\ $\lk(\varepsilon_i,\varepsilon^+_i)=0$). Suppose again that $\{p_1, p_2, p_3\}$ is a triple of integers such that $|\frac{n_i}{m_i}|>1$ and $p_i\cdot(p_i-1) \neq 0$ for $i = 1,2,3$, then by similar analysis to that in Corollary~\ref{corollary:eightmetabolisers}, and again using Proposition~\ref{proposition:possiblemetaboliser}, it is possible to deduce that there are $8$ possible metabolisers of the form
$$H = \spn\{ [L_1], [L_2], [L_3] \} \text{ where } L_i \in \{J_i,\epsilon_i\} \text{ for } i=1,2,3.$$
By the same argument as in Example~\ref{example1}, $\psi(K,P) \neq 0$ for all possible lagrangians $P \subset \AK$.

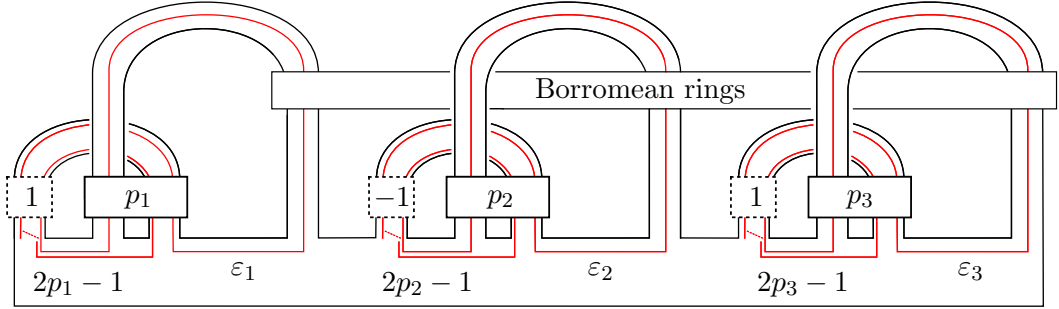
\begin{figure}[htbp]
\centering
\begin{picture}(400,140)
\put(0,0){\includegraphics[width=5.5in]{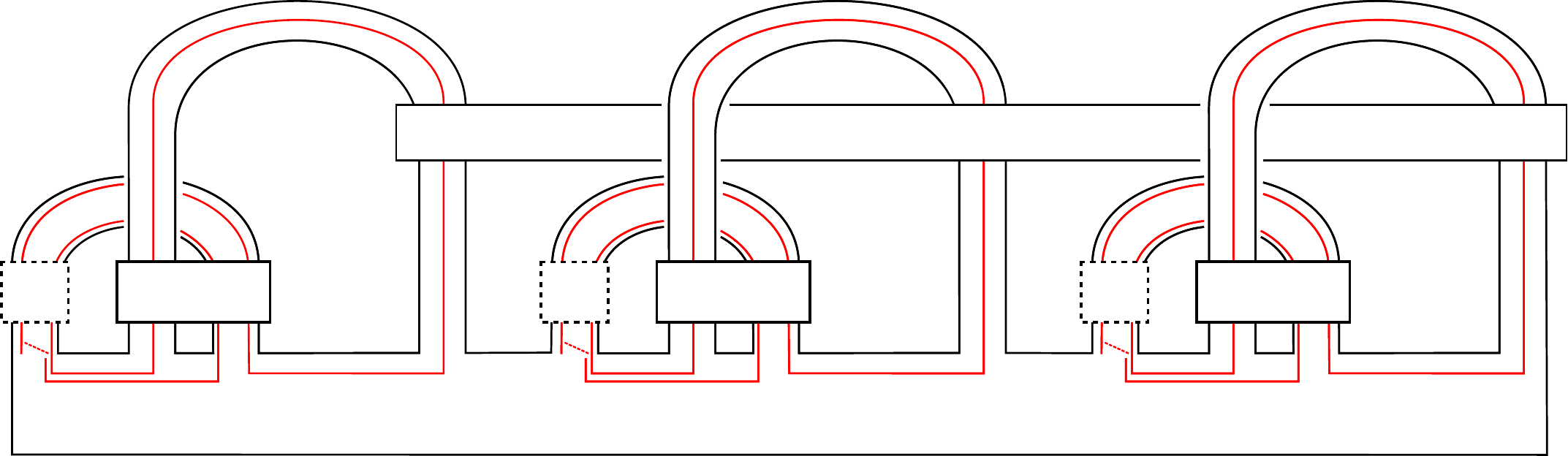}}
\put(200,79){Borromean rings}
\put(7,38){$1$}
\put(139,38){$-1$}
\put(281,38){$1$}
\put(45,39.5){$p_1$}
\put(182,39.5){$p_2$}
\put(318,39.5){$p_3$}
\put(10,6){$2p_1-1$}
\put(142,6){$2p_2-1$}
\put(284,6){$2p_3-1$}
\put(85,12){$\varepsilon_1$}
\put(220,12){$\varepsilon_2$}
\put(360,12){$\varepsilon_3$}
\end{picture}
\caption{A knot $K$, with a disc-band form for a Seifert surface $\Sigma$, where a solid box represents $p_i$ full twists between two bands with no twist on each bands, and a dotted box represents $\pm 1$ full twists between two strands. For $i\in\{1,2,3\}$, $\varepsilon_i$ is a simple closed curve that goes around $(2i-1)$-th band $2p_i-1$ times, and around the $2i$-th band $-1$ times. In the box labelled Borromean rings, the bands are tied in a string link whose closure is the Borromean rings, without introducing extra twisting.}
\label{figure:example3}
\end{figure}
\end{example}

\subsection{The doubly solvable filtration}
In this subsection we give the proof of Theorem~\ref{theorem:doubly-slice-theorem}. First, we recall the results that are previously known.  The next theorem was shown in \cite[Theorem 1.1 and Theorem 7.1]{Kim06}.

\begin{theorem}\label{thm:kim06}~
\begin{enumerate}
\item For a given integer $n\geq 1$, there exists a ribbon knot $K$ that is algebraically doubly slice, doubly $(n)$-solvable, but not doubly $(n.5)$-solvable.
\item For a given integer $n\geq 1$, there exists an algebraically doubly slice knot $K$ that is doubly $(n)$-solvable, but not $(n,n.5)$-solvable.
\end{enumerate}
\end{theorem}

\noindent Taehee Kim also showed that first few terms of the doubly solvable filtration are well understood~\cite[Proposition 2.8 and Proposition 2.10]{Kim06} (see also \cite[Section 7]{Cha-Kim:2016-1}, \cite{Orson:2017-1}).

\begin{proposition}\label{prop:kim06}~
\begin{enumerate}
\item For $n=0$ or $0.5$, a knot $K$ is doubly $(n)$-solvable if and only if it is $(n)$-solvable. Hence, a knot is doubly $(0)$-solvable if and only if it has vanishing $\Arf$ invariant, and doubly $(0.5)$-solvable if and only if it is algebraically slice.
\item \label{item:alg-doubly-slice} If a knot $K$ is doubly $(1)$-solvable, then $K$ is algebraically doubly slice.
\end{enumerate}
\end{proposition}

Theorem~\ref{theorem:doubly-slice-theorem} shows that (a weaker form of) the converse of Proposition~\ref{prop:kim06}~(\ref{item:alg-doubly-slice}) does not hold. Theorem~\ref{theorem:doubly-slice-theorem} is analogous to  Theorem~\ref{thm:kim06} for the base case of the ``other half'' of the filtration. We recall the statement of Theorem~\ref{theorem:doubly-slice-theorem} for the convenience of the reader.

\medskip
\begin{theoremA}~
\begin{enumerate}[(a)]
{\setlength\itemindent{15pt}\item\label{item:thmA-a} There exists a ribbon knot that is algebraically doubly slice, but not doubly $(1)$-solvable.}
{\setlength\itemindent{15pt}\item\label{item:thmA-b} There exists a knot that is algebraically doubly slice, but not $(0.5,1)$-solvable.}
\end{enumerate}
In particular, neither knot is doubly slice.
\end{theoremA}

\begin{proof}
For part (\ref{item:thmA-a}), let $K$ be the knot from Example~\ref{example1}, for some choice of $p_1$, $p_2$ and $p_3$,
 except that we do not tie Borromean rings into the $1$st and $3$rd and $5$th bands (that is when $L_i = \delta_i$ for $i=1,2,3$). The derivative $\delta_1 \cup \delta_2 \cup \delta_3$ of $K$ is an unlink, which implies that $K$ is a ribbon knot. The Seifert form with respect to the given basis is \[\bp 0 & X \\ X - \Id & 0 \ep,\] so $K$ is algebraically doubly slice, and Proposition~\ref{prop:kim06}~$(1)$ implies that $K$ is doubly $(0.5)$-solvable. We observed in Example~\ref{example1} that $K$ has eight possible metabolisers, and exactly one of the metabolisers, namely $\spn\{ [\delta_1], [\delta_2],[\delta_3] \}$, represents a lagrangian $P_0$ with respect to which $\psi(K,P_0) = 0$, using Corollary~\ref{corollary:TFAE}. If $K$ were doubly $(1)$-solvable, then by Theorem~\ref{theorem:doubly-slice-obstruction}, there would exist two lagrangians $P_1$ and $P_2$ for the rational Blanchfield form, such that $P_1 \oplus P_2 = H_1(M_K;\Q[\Z])$ and $\psi(K,P_1) = \psi(K,P_2) = 0$. This contradicts the statement above that there is exactly one lagrangian $P$ with $\psi(K,P)=0$. This concludes the proof of part (\ref{item:thmA-a}).

For the second part, let $K$ be the knot from Example~\ref{example1}. For the same reason as above, $K$ is algebraically doubly slice and doubly $(0.5)$-solvable. If $K$ were $(0.5,1)$-solvable, then by Theorem~\ref{theorem:doubly-slice-obstructionhalf}, there would exist a lagrangian $P$ for the rational Blanchfield form such that $\psi(K,P) = 0$.
 This is not possible, since we checked in Example~\ref{example1} that $\psi(K,P) \neq 0$ for all possible lagrangians $P$.  This concludes the proof of part~(\ref{item:thmA-b}) and therefore of the theorem.
\end{proof}

\noindent We end this section with the following observations.
\medskip

\begin{lemma}~\label{lemma:0.5ribbon}
A ribbon knot $K$ is $(0.5,n)$-solvable for all $n\in \frac{1}{2}\mathbb{N}_0$.
\end{lemma}

\begin{proof}
Since $K$ is $(0.5)$-solvable there exists a $(0.5)$-solution $W_{0.5}$ with $\pi_1(W_{0.5})=\mathbb{Z}$~\cite[Remark 1.3]{COT03}. Let $W_R$ be the ribbon disc complement for $K$. Then by the Seifert-van Kampen theorem, it is straightforward to conclude that $\pi_1(W_{0.5}\cup_{M_K}W_R) = \mathbb{Z}$. Therefore $K$ is $(0.5,n)$-solvable for all $n\in \frac{1}{2}\mathbb{N}_0$.
\end{proof}

\begin{corollary}
There exists a knot $K$ that is algebraically doubly slice and $(0.5,n)$-solvable for all $n\in \frac{1}{2}\mathbb{N}_0$, but is not doubly $(1)$-solvable. \end{corollary}

\begin{proof}
Let $K$ be a ribbon knot from Theorem~\ref{theorem:doubly-slice-theorem}~(\ref{item-theorem-A-1}). Then by Lemma~\ref{lemma:0.5ribbon}, $K$ is $(0.5,n)$-solvable for all $n\in \frac{1}{2}\mathbb{N}_0$. We proved that $K$ is not doubly $(1)$-solvable in Theorem~\ref{theorem:doubly-slice-theorem}.
\end{proof}

\subsection{Algebraically slice knots with potentially interesting properties}

In this subsection, we investigate some properties of knots that are algebraically slice and have non-vanishing $\mathbb{Z}[\mathbb{Z}]$ homology ribbon obstruction.   First we define what is means for a knot to be homotopy ribbon $(n)$-solvable, and then we recall a generalised version of the Kauffman conjecture. We show that there exists an algebraically slice knot $K$ that is not homotopy ribbon $(1)$-solvable and does not have any $(0)$-solvable derivative. At the end of the section, we present some interesting properties of a set, denoted $\mathcal{S}$, of algebraically slice knots that do not have vanishing $\Z[\Z]$ homology ribbon obstruction (see Definition~\ref{defn:hrn-solvable}).

Motivated by the definition of a homotopy ribbon knot, we can define the following analogous definition for the solvable filtration. It is not known whether every $(n)$-solvable knot is homotopy ribbon $(n)$-solvable or not.

\medskip
\begin{definition}[{Homotopy ribbon $(n)$-solvable}]~\label{defn:hrn-solvable}
 We say that a knot $K$ is \emph{homotopy ribbon $(n)$-solvable} for $n \in \frac{1}{2}\mathbb{N}_0$ if the zero-framed surgery manifold $M_K$ is the boundary of an $(n)$-solution $W$ such that the inclusion induced map $\pi_1(M_K) \to \pi_1(W)$ is surjective.
\end{definition}

We note that a knot $K'$ concordant to $K$ need not be homotopy ribbon $(n)$-solvable even if $K$ is, just as the ordinary homotopy ribbon property need not be preserved under concordance.

We recall an open problem.  Note that if a knot has a slice derivative then the knot itself is a slice knot. It is natural to ask if the converse is true as follows (see also \cite[Conjecture 7.2]{CD14}).

\medskip
\begin{conjecture}[{Generalised version of the Kauffman Conjecture}]~\label{conj:genKauffman}
If $K$ is a topologically (resp.\ smoothly) slice knot, then there exist a topologically (resp.\ smoothly) slice derivative of $K$.
\end{conjecture}

As mentioned in the introduction, every ribbon knot has a Seifert surface with an unlinked derivative. Hence if slice-ribbon conjecture holds, then the smooth version of Conjecture~\ref{conj:genKauffman} also holds. In \cite{CD14}, Cochran and Davis found a smoothly slice knot $R$, where $R$ has a unique minimal genus one Seifert surface $F$, but there does not exist any slice derivative on $F$. However, the smoothly slice knot $R$ in \cite{CD14} can be shown to be ribbon by finding a ribbon derivative after stabilising the Seifert surface $F$. It is also known that if a knot has an $(n)$-solvable derivative then the knot itself is $(n+1)$-solvable~\cite[Theorem 8.9]{COT03}. We ask whether the converse is true cf.~\cite[Conjecture 1.4]{CD15}).

\medskip
\begin{conjecture}[{$(n)$-solvable Kauffman Conjecture}]~\label{conj:ngenKauffman}
For all $n\in\frac{1}{2}\mathbb{N}_0$, if $K$ is $(n+1)$-solvable, then there exist a $(n)$-solvable derivative of $K$.
\end{conjecture}

We do not have a counter example for Conjecture~\ref{conj:ngenKauffman}. But we show that the knot from Example~\ref{example1} has the following interesting property: whether or not this knot is $(1)$-solvable, it is not possible to show that it is $(1)$-solvable by finding a $(0)$-solvable derivative.

\begin{theoremC}
There exist an algebraically slice knot $K$ that is not homotopy ribbon $(1)$-solvable and does not have any $(0)$-solvable derivative.
\end{theoremC}

%
%
%
%

\begin{proof}[Proof of Theorem~\ref{thm:any-seifert-surface-non-trivial-mu-123}]
 Let $K$ be the knot from Example~\ref{example1}. Note that every homotopy ribbon $(1)$-solvable knot is homology ribbon $(1)$-solvable (see Lemma~\ref{lemma:homotopyribbonsol}). Hence by Theorem~\ref{theorem:obstruction}, $K$ is not a homotopy ribbon $(1)$-solvable knot.  Now, suppose that $K$ has a $(0)$-solvable derivative~$J$ with $m$ components. Then $M_J$ bounds over $\Z^m$ and so $\bar{\mu}_{J}(ijk)=0$ for any subset $\{i,j,k\}$ of the indexing set for the components of $J$ cf.~\cite{Mar15}. However $\psi(K,P) \neq 0$ for all lagrangians, which by Theorem~\ref{thm:vanishesiftriple} implies that $\bar{\mu}_{J}(ijk) \neq 0$ for some triple $(ijk)$.
\end{proof}

\noindent We end this section by presenting some interesting properties of the following set.

\begin{definition}\label{defn:set-mathcal-S}
Let $\mathcal{S}$ be the set of all algebraically slice knots $K$ such that the invariant $\psi(K,P) \neq 0$ for all possible lagrangians $P \subset \AK$.
\end{definition}

Recall that there is a bipolar filtration of $\mathcal{C}$, defined by Cochran, Harvey and Horn in \cite{CHHo13}, that generalises the notion of positivity from \cite{CG88}. We refer to \cite{CHHo13} for the definition and detailed discussion.

In the upcoming proposition, $\tau$ denotes the concordance invariant of Ozv\'ath-Szab\'o \cite{OzSz03}, $s$ denotes the concordance invariant of Rasmussen \cite{Ra10}, $d_1$ denotes the concordance invariant of Peters \cite{Pe10} where $d_1(K)=d(S_1^3(K))$ is the correction term defined by Ozv\'ath-Szab\'o \cite{OzSz03b}, $S_1^3(K)$ denotes the one-framed surgery on $S^3$ along~$K$, $\nu^+$ denotes the concordance invariant of Hom-Wu \cite{HomWu16}, and finally $\Upsilon$ denotes the concordance invariant of Ozv\'ath-Stipsicz-Szab\'o \cite{OSS17}. Note that all the above invariants obstruct a knot from being smoothly slice, and indeed from being $0$-bipolar.

\begin{proposition}\label{prop:propertyS}
Let $\mathcal{S}$ be the set of knots from Definition~\ref{defn:set-mathcal-S}.
\begin{enumerate}
\item\label{propertyS:item1} There are infinitely many concordance classes of knots in $\mathcal{S}$.
\item\label{propertyS:item2} For any $K \in \mathcal{S}$, $K$ does not have a $(0)$-solvable derivative. In particular, no derivative of $K$ is topologically slice.
\item\label{propertyS:item3} For any $K \in \mathcal{S}$, $K$ is not homotopy ribbon $(1)$-solvable. In particular, $K$ is not homotopy ribbon.
\item\label{propertyS:item4} There exists $K \in \mathcal{S}$ such that $K \notin \mathcal{F}_{1.5}$
\item \label{propertyS:item5}There exists a knot $K \in \mathcal{S}$ such that $K \in \mathcal{B}_0$ where $\mathcal{B}_0$ is the set of $0$-bipolar knots. In particular, $\tau(K)=s(K)=d_1(K)=\varepsilon(K)=\nu^+(K)=\Upsilon(K)=0$.
\item\label{propertyS:item6} If there exists a knot $K\in \mathcal{S}$ that is smoothly slice, then $K$ gives a counterexample for ribbon-slice conjecture.
\item\label{propertyS:item7} If there exists a knot $K\in \mathcal{S}$ that is topologically slice, then $K$ gives a counterexample for homotopy ribbon-slice conjecture.
\end{enumerate}
\end{proposition}

\begin{proof}
To prove~(\ref{propertyS:item1}), consider knots with the same Seifert form as in Example~\ref{example1}. First, we show that there exist infinitely many triples of integers $\{p_1, p_2, p_3\}$ such that corresponding $|\frac{n_i}{m_i}|>1$ and $p_i\cdot(p_i-1) \neq 0$ for $i\in \{1,2,3\}$. This can be achieved by letting $p_1 = 2^n+1$, $p_2=2^{2n}+1$, $p_3=2^{4n}+1$, since
$$\begin{pmatrix}
\frac{n_1}{m_1}\\
\frac{n_2}{m_2}\\
\frac{n_3}{m_3}\\
\frac{n_4}{m_4}\end{pmatrix} = \begin{pmatrix}
(2^n+1)(2^{2n}+1)(2^{4n}+1)-2^{7n}\\
2^{3n}+2^{2n}+2^n-1\\
2^{4n}-2^{2n}+2^n+1\\
-2^{5n}+2^{4n}+2^{2n}+1\end{pmatrix}.$$
Choose two triples $\{p_1, p_2, p_3\}$ and $\{p_1', p_2', p_3'\}$ with above property where $p_1, p_2, p_3,p_1', p_2', p_3'$ are all distinct. Let $K$ and $K'$ be knots from Example~\ref{example1} where $\{p_1, p_2, p_3\}$ corresponds to $K$ and $\{p_1', p_2', p_3'\}$ corresponds to $K'$. If they are concordant then by \cite[Corollary $5.9$]{CHL10} $K \# -K'$ should have a derivative with bounded von Neumann $\rho$-invariant of its zero surgery.
For every metaboliser, perform infection on $K'$ to make the von Neumann $\rho$-invariant of zero surgery on the derivatives of $K \# -K'$ larger than upper bound given by \cite[Corollary~10.2]{CHL10}.  This guarantees that $K$ and $K'$ are not concordant, and we can repeat this process to get infinitely many different concordance classes of knots in $\mathcal{S}$.
%

Items~(\ref{propertyS:item2}) and (\ref{propertyS:item3}) follow from the proof of Theorem~\ref{thm:any-seifert-surface-non-trivial-mu-123}.

For~(\ref{propertyS:item4}), it is known that if $K \in \mathcal{F}_{1.5}$, then there is an upper bound for the von Neumann $\rho$-invariant of zero surgery on a derivative of $K$ that represents particular metaboliser of~$K$ ~\cite[Corollary~$10.2$]{CHL10}. The effect of infection on the von Neumann $\rho$-invariant  is well understood ~\cite[Proposition~$3.2$]{COT04},~\cite[Lemma~$2.3$]{CHL09}. For instance, if we take a knot from Example~\ref{example1} and infect each of the bands enough (for example, by tying in a connected sum of many trefoils) so that von Neumann $\rho$-invariant of zero surgery on the derivatives of $K$ for each metaboliser becomes larger than the upper bound given by \cite[Corollary~$10.2$]{CHL10}, then we can guarantee that $K$ is not $(1.5)$-solvable. Hence $(4)$ holds.

For~(\ref{propertyS:item5}), we will use Example~\ref{example2}. Note that the knot from Example~\ref{example2} can be turned into a slice knot by changing a positive crossing to a negative crossing (undo the positive crossing on the third band). Also this knot can be turned into a slice knot by changing a negative crossing to a positive crossing (undo the negative crossing on the first band). Whence $K \in \mathcal{B}_0$ ~\cite[Lemma~3.4]{CL86}, \cite[Proposition~3.1]{CHHo13} and the result follows from \cite{CHHo13,NW14,HomWu16,OSS17}.

Items~(\ref{propertyS:item6}) and ~(\ref{propertyS:item7}) follow immediately from Theorem~\ref{ribbon-obstruction-theorem}.\end{proof}

\bibliographystyle{alpha}
\bibliography{knotbib}

\end{document}